\documentclass[12pt]{amsart}
\usepackage{amssymb,amsmath,amsfonts,amsthm,bbm,mathrsfs,times,graphicx,color,comment,mathabx,enumerate}
\usepackage[utf8]{inputenc}
\usepackage{tikz}
\usetikzlibrary{arrows,shapes}
\usetikzlibrary{fadings}
\usetikzlibrary{intersections}
\usetikzlibrary{decorations.pathreplacing}

\usepackage{mathtools}

\DeclarePairedDelimiter\floor{\lfloor}{\rfloor}

\DeclareMathOperator{\re}{\mathbb{R}e}

\newcommand{\INF}{{\infty}}

\newcommand{\dist}{\mbox{dist}}
\newcommand{\tta}{\theta}

\newcommand{\OM}{\Omega}

\newcommand{\sph}{{{\mathbf S}^ 1}}

\newcommand{\del}{\partial}

\newcommand{\Gam}{\varGamma}
\newcommand{\ol}{\overline}
\newcommand{\ds}{\displaystyle}
\newcommand{\dba}{\overline{\partial}}

\newcommand{\BR}{\mathbb{R}}

\newcommand{\BZ}{\mathbb{Z}}
\newcommand{\BN}{\mathbb{N}}
\newcommand{\fii}{{\varphi}}
\newcommand{\bu}{{\bf u}}
\newcommand{\bv}{{\bf v}}

\newcommand{\bg}{{\bf g}}
\newcommand{\bq}{{\bf q}}
\newcommand{\bbf}{{\bf f}}

\newcommand{\btheta}{\boldsymbol \theta}
\newcommand{\balpha}{\boldsymbol \alpha}
\newcommand{\bbeta}{\boldsymbol \beta}

\newcommand{\spr}[2]{\langle #1,#2 \rangle}
\newcommand{\lnorm}[1]{ \left\| #1 \right\|}

\newcommand{\commentK}[1]{\par\noindent\textcolor{blue}{\textbf{Kamran: \textit{#1}}}\par}
\newcommand{\Laplace}{\mathop{}\!\mathbin\bigtriangleup}
\newtheorem{theorem}{Theorem}[section]
\newtheorem{prop}{Proposition}[section]
\newtheorem{lemma}{Lemma}[section]
\newtheorem{cor}{Corollary}[section]

\title[A Fourier approach to the inverse source problem]{A Fourier approach to the inverse source problem in an absorbing and anisotropic scattering medium}
\begin{document}
\date{\today}
\author{Hiroshi Fujiwara}
\address{Graduate School of Informatics,  Kyoto University, Yoshida Honmachi, Sakyo-ku, Kyoto 606-8501, Japan }
\email{fujiwara@acs.i.kyoto-u.ac.jp}
\author{Kamran Sadiq}
\address{Johann Radon Institute of Computational and Applied Mathematics (RICAM), Altenbergerstrasse 69, 4040 Linz, Austria}
\email{kamran.sadiq@ricam.oeaw.ac.at}
\author{Alexandru Tamasan}
\address{Department of Mathematics, University of Central Florida, Orlando, 32816 Florida, USA}
\email{tamasan@math.ucf.edu}

\maketitle

\begin{abstract}
We revisit the  inverse source problem in a two dimensional absorbing and scattering medium and present a non-iterative reconstruction method using measurements of the radiating flux at the boundary. The attenuation and scattering coefficients are known and the unknown source is isotropic. The approach is based on the Cauchy problem for a Beltrami-like equation for the sequence valued maps, and extends the original ideas of A. Bukhgeim  from the non-scattering to scattering media. We demonstrate the feasibility of the method in a numerical experiment in which the scattering is modeled by the two dimensional Henyey-Greenstein kernel with parameters meaningful in Optical Tomography.  
\end{abstract}
\section{Introduction}

This work concerns a Fourier approach to the inverse source problem for radiative transport in a strictly convex domain $\OM$ in the Euclidean plane. The attenuation and scattering coefficients are known real valued functions. Generated by an unknown source $f$, in the steady state case, the density of particles $u(z,\btheta)$ at $z$ traveling in the direction $\btheta$ solve the stationary transport equation
\begin{align} \label{TransportScatEq1}
\btheta\cdot\nabla u(z,\btheta) +a(z) u(z,\btheta) &=  \int_{\sph} k(z,\btheta\cdot\btheta')u(z,\btheta') d\btheta' + f(z) , \quad (z,\btheta)\in \OM\times\sph,
\end{align}where 
$\sph$ denotes the unit sphere. 

Let $\Gam_\pm:=\{(z,\btheta)\in \Gam \times\sph:\, \pm\nu(z)\cdot\btheta>0 \}$ be the incoming (-), respectively outgoing (+), unit tangent sub-bundles of the boundary; where  $\nu(z)$ is the outer unit normal at $z \in \Gam$.
The (forward) boundary value problem for \eqref{TransportScatEq1} assumes a given  incoming flux $u$ on ${\Gam_-}$, In here we assume that there is no incoming radiation from outside the domain, $\displaystyle u|_{\Gamma_-}=0$. The boundary value problem is know to be well-posed under various admissibility and subcritical assumptions, e.g, in \cite{dautrayLions4, choulliStefanov96, choulliStefanov99, anikonov02,mokhtar}, with the most general result for a generic pair of coefficients obtained by Stefanov and Uhlmann \cite{stefanovUhlmann08}. In here we assume that the forward problem is well-posed, and that the outgoing radiation $u |_{\Gamma_+}$ is measured, and thus the trace $u \lvert_{\Gamma\times\sph}$ is known. Without loss of generality $\Omega$ is the unit disc.

In here we show how to recover $f$ from knowledge of  $u$ on the torus $\Gamma\times\sph$ and provide an error, and stability estimates.

When $a=k=0$, this is the classical $X$-ray tomography problem of Radon \cite{radon1917}, where $f$ is to be recovered from its integrals along lines, see also \cite{naterrerBook,helgasonBook,ludwig}. For $a\neq 0$ but $k=0$, this is the problem of inversion of the Attenuated Radon transform in two dimensions, solved successfully by Arbuzov, Bukhgeim and Kazantsev ~\cite{ABK}, and Novikov \cite{novikov01}; see \cite{naterrer01,bomanStromberg,bal04} for later approaches. 

The inverse source problem in an absorbing and scattering media, $a,k\neq0$, has also been considered (e.g., \cite{larsen75, siewert93}) in the Euclidean setting, and in \cite{sharafutdinov97} in the Riemannian setting. The most general result ($k$ may vary with two independent directions) on the stable determination of the source was obtained by Stefanov and Uhlmann \cite{stefanovUhlmann08}. The reconstruction of the source based on \cite{stefanovUhlmann08} is yet to be realized. When the anisotropic part of scattering is sufficiently small, a convergent iterative method for source reconstruction was proposed in \cite{balTamasan07}. Based on a perturbation argument to the non-scattering case in \cite{novikov01}, the method does not extend to strongly anisotropic scattering. In addition, it requires solving one forward problem  (a computationally extensive effort) at each iteration.

The main motivation of this work is to provide a source reconstruction method that applies to the anisotropic scattering media, with non-small anisotropy. In here we propose such a non-iterative method. Our approach extends the original ideas in \cite{ABK} from the non-scattering to the scattering media. 

Throughout we assume that $a\in C^{2,s}(\ol\OM)$, and $k$ and its angular derivative are periodic in the angular variable,  $k\in C^{1,s}_{per}([-1,1]; Lip(\Omega))$, $s>1/2$, and that the forward problem is well posed. It is known from \cite{stefanovUhlmann08} that for pairs of coefficients $(a,k)$ in an open and dense sets of $C^2 \times C^2$, and for any $f\in L^2(\Omega)$, there is a unique solution $u\in L^2(\OM\times\sph)$ to the forward boundary value problem. However, our approach requires a smooth solution $u\in H^1(\Omega\times\sph)$. As a direct consequence of  \cite[Proposition 3.4]{stefanovUhlmann08} the regularity of $u$ is dictated by its ballistic term. In particular, if $f\in H^1(\Omega)$, then  $u\in H^1(\OM\times\sph)$. With the exception of the numerical examples in Section \ref{Sec:numerics}, we assume that the unknown source $f\in H^1(\Omega)$, and thus the unknown solution 
\begin{equation}\label{regularity}
u\in H^1(\Omega\times\sph).\end{equation}
In the numerical experiment we use a discontinuous source, whose successful quantitative reconstruction indicates robustness of the method.

Let $u(z,\btheta) = \sum_{-\infty}^{\infty} u_{n}(z) e^{in\tta}$ be the formal Fourier series representation of $u$ in the angular variable $\btheta=(\cos\tta,\sin\tta)$. Since $u$ is real valued, $u_{-n}=\ol{u_n}$ and the angular dependence is completely determined by the sequence of its nonpositive Fourier modes 
\begin{align}\label{boldu}
\OM \ni z\mapsto  \bu(z)&: = \langle u_{0}(z), u_{-1}(z),u_{-2}(z),... \rangle.
\end{align}
Let $k_n(z)=\frac{1}{2\pi} \int_{-\pi}^{\pi} k(z,\cos \theta ) e^{-i n\theta}d\theta$, $n\in\mathbb{Z}$, be the Fourier coefficients of the scattering kernel. Since $k(z, \cos \tta)$ is both real valued and even in $\tta$,  $k_n(z)$ are real valued and $k_n(z)=k_{-n}(z)$.

Throughout this paper the Cauchy-Riemann operators $\dba = (\del_x+i\del_y)/2$ and $\del =(\del_x-i\del_y)/2$ refer to derivatives in the spatial domain. By using the advection operator $\btheta \cdot\nabla=e^{-i\tta}\dba + e^{i\tta}\del$, and identifying the Fourier coefficients of the same order, the equation \eqref{TransportScatEq1} reduces to the system:
\begin{align}\label{freeeq}
\overline{\del} u_{1}(z)+\del u_{-1}(z) + a(z)u_{0}(z)=k_{0}(z)u_{0}(z)+f(z),
\end{align}and 
\begin{align}\label{infsys0}
\dba u_n(z) +\del u_{n-2}(z) + a(z)u_{n-1}(z)=k_{n-1}(z)u_{n-1}(z),\quad n\neq 1.
\end{align}
In particular, the sequence valued map \eqref{boldu} solves the Beltrami-like equation
\begin{align}\label{beltrami}
\dba\bu(z) +L^2 \del\bu(z)+ a(z)L\bu(z) = L K\bu(z),\quad z\in \OM,
\end{align}
where  $L\bu(z)=L (u_0(z),u_{-1}(z),u_{-2}(z),...):=(u_{-1}(z),u_{-2}(z),...)$ denotes the left translation, and
\begin{align}\label{multiplier}
K\bu(z):=(k_0(z) u_0(z),k_{-1}(z)u_{-1}(z),k_{-2}(z)u_{-2}(z),...)
\end{align}is a Fourier multiplier operator determined by the scattering kernel.

Our data $u \lvert_{\Gam \times \sph}$ yields the trace of  a solution of \eqref{beltrami} on the boundary,  
\begin{align}\label{gdata_defn}
 \bg = \bu \lvert_{\Gam} = \langle g_{0}, g_{-1}, g_{-2},... \rangle.
\end{align}

Bukhgeim's original  theory in \cite{bukhgeimBook}  concerns solutions of \eqref{beltrami} for $a=0$ and $K=0$. Solutions of
\begin{align}\label{Analytic}
\dba \bu+ L^2\del \bu=0,
\end{align}(called $L^2$-analytic) satisfy a Cauchy-like integral formula, which recovers $\bu$ in $\OM$ from its trace $\bu |_{\Gam}$. In the explicit form in \cite{finch}, for each $n \geq 0$,
\begin{align}\label{CauchyBukhgeimformula}
u_{-n}(\zeta)=\frac{1}{2\pi i}\int_\Gam\frac{u_{-n}(z)}{z-\zeta}dz + 
\frac{1}{2\pi i} \int_\Gam \left\{
\frac{dz}{z-\zeta} - \frac{d\ol{z}}{\ol{z}-\ol{\zeta}}
\right\}\sum_{j=1}^\infty u_{-n-2j}(z)\left(\frac{\ol{z-\zeta}}{z-\zeta}\right)^j,\;\zeta\in\OM.
\end{align}

In Section \ref{Sec:prelim} we review the absorbing, non-scattering case. While we follow the treatment in \cite{sadiqTamasan01}, it is in this section that the new analytical framework and notation is introduced. Section \ref{Sec:Sourcerec_poly} describe the reconstruction method for scattering kernels of polynomial dependence in the angular variable. Except for the numerical section in the end, the remaining of the paper analyzes the error made by the polynomial approximation of the scattering kernel. In Section \ref{Sec:Gradientest_BBeq} we exhibit the gain in smoothness due to the scattering, in particular, the $1/2$-gain in \eqref{LNKvphalf} below has been known (with different proofs) see \cite{mokhtar}, and in a more general case than considered here in \cite{stefanovUhlmann08}.

The key ingredient in our analysis is an a priori gradient estimate  for solutions of the inhomogeneous Bukhgeim-Beltrami equations, see Theorem \ref{new_estimate} in Section \ref{Sec:Gradientest_BBeq}. Our starting point  is an energy identity, an idea originated in the work of Mukhometov \cite{mukhometov75}, and an equivalent of Pestov's identity \cite{sharafutdinov97,tamasan03} for the Bukhgeim-Beltrami equation. The proof of the gradient estimate in Theorem \ref{new_estimate} uses essentially $1$ derivative gain in smoothness  due to scattering. As a consequence, in Theorem \ref{error_estimate} we establish an error estimate, which yields a stability result for scattering with polynomial angular dependence (see Corollary \ref{Cor_errorEst}). Furthermore, in a weakly anisotropic scattering medium the method is convergent (see Theorem \ref{convergence_result}),  thus recovering the result in \cite{balTamasan07}.

The feasibility of the proposed method is implemented in two numerical experiments in Section \ref{Sec:numerics}. 
Among the several models for the scattering kernel used in Optical Tomography \cite{arridge}, we work with the two dimensional version of the Henyey-Greenstein kernel for its simplicity.
In this kernel we chose the anisotropy parameter to be $1/2$ ( half way between the ballistic and isotropic regime), and a mean free path of  $1/5$ units of length e.g. (meaningful value for fluorescent light scattering in  the fog).


\section{A brief review of the absorbing non-scattering medium}\label{Sec:prelim}

In the case for $a\neq0$ and $K=0$, the Beltrami equation \eqref{beltrami} can be reduced to \eqref{Analytic} via an integrating factor. While this idea originates in \cite{ABK},
in here (as in \cite{sadiqTamasan01}) we use the special integrating factor proposed by Finch in \cite{finch}, which enjoys the crucial property of having vanishing negative Fourier modes.
This special integrating factor is $e^{-h}$, where 
\begin{align}\label{hDefn}
h(z,\btheta) := Da(z,\btheta) -\frac{1}{2} \left( I - i H \right) Ra(z\cdot \btheta^{\perp}, \btheta^{\perp}),
\end{align}and $\btheta^\perp$ is  orthogonal  to $\btheta$, 
$Da(z,\btheta) =\ds \int_{0}^{\INF} a(z+t\btheta)dt$ is the divergent beam transform of the attenuation $a$, 
$Ra(s, \btheta^{\perp}) = \ds \int_{-\INF}^{\INF} a\left( s \btheta^{\perp} +t \btheta \right)dt$ is the Radon transform of
the attenuation $a$, and
the classical Hilbert transform $H h(s) = \ds \frac{1}{\pi} \int_{-\INF}^{\INF} \frac{h(t)}{s-t}dt $ is 
taken in the first variable and evaluated at $s = z \cdotp \btheta^{\perp}$. 
The function $h$  appeared first in the work of Natterer \cite{naterrerBook}; see also \cite{bomanStromberg} for elegant arguments that show how $h$ extends analytically (in the angular variable on the unit circle $\sph$) inside the unit disc.  We recall some properties of $h$ from \cite[Lemma 4.1]{sadiqScherzerTamasan}, while establishing notations.
\begin{lemma}\cite[Lemma 4.1]{sadiqScherzerTamasan}\label{hproperties}
Assume $\OM \subset \BR^2$ is $ C^{2,s}, s>1/2$, convex domain. For $p = 1,2$, let $a\in C^{p,s}(\ol \OM)$, $s>1/2$, and $h$ defined in \eqref{hDefn}. Then $h \in C^{p,s}(\ol \OM \times \sph)$ and the following hold \\
(i) $h$ satisfies \begin{align}\label{h_IntegratingFactor}
\btheta \cdot \nabla h(z,\btheta) = -a(z), \quad (z, \btheta) \in \OM \times \sph.
\end{align}
(ii) $h$ has vanishing negative Fourier modes yielding the expansions
\begin{align}\label{ehEq}
  e^{- h(z,\tta)} := \sum_{k=0}^{\INF} \alpha_{k}(z) e^{ik\fii}, \quad e^{h(z,\tta)} := \sum_{k=0}^{\INF} \beta_{k}(z) e^{ik\fii}, \quad (z, \tta) \in \ol\OM \times \sph,
\end{align}
with (iii)
\begin{align}\label{balpha_seq}
&z\mapsto \balpha(z) := \langle \alpha_{0}(z), \alpha_{1}(z), \alpha_{2}(z), \alpha_{3}(z), ... , \rangle \in C^{p,s}(\OM ; l_{1})\cap C(\ol\OM ; l_{1}), \\ \label{bbeta_seq}
&z\mapsto \bbeta(z) := \langle \beta_{0}(z), \beta_{1}(z), \beta_{2}(z), \beta_{3}(z), ... , \rangle \in C^{p,s}(\OM ; l_{1})\cap C(\ol\OM ; l_{1}).
\end{align}
(iv) For any $z \in \OM $
\begin{align}\label{betazero_one}
&\ol{\del} \beta_0(z) = 0, \qquad \ol{\del} \beta_1(z) = -a(z) \beta_0(z),\\ \label{betak}
& \ol{\del} \beta_{k+2}(z) +\del \beta_{k}(z) +a(z) \beta_{k+1}(z)=0, \quad k \geq 0.
\end{align}
(v) For any $z \in \OM $
\begin{align}\label{alphazero_one}
&\ol{\del} \alpha_0(z) = 0, \qquad \ol{\del} \alpha_1(z) = a(z) \alpha_0(z),\\ \label{alphak}
& \ol{\del} \alpha_{k+2}(z) +\del \alpha_{k}(z) -a(z) \alpha_{k+1}(z)=0, \quad k \geq 0.
\end{align}
(vi) The Fourier modes $\alpha_{k}, \beta_{k}, k\geq 0$ satisfy
\begin{align}\label{alphabetaSys}
\alpha_0 \beta_0 =1, \quad \sum_{m=0}^{k} \alpha_{m}\beta_{k-m}=0, \quad k \geq 1.
\end{align}
\end{lemma}

The Fourier coefficients of  $e^{\pm h}$  define the integrating operators $e^{\pm G} \bu $ component-wise for each $m \leq 0$ by 
\begin{align}\label{eGop}
(e^{-G} \bu )_m = (\balpha \ast \bu)_m = \sum_{k=0}^{\infty}\alpha_{k} u_{m-k}, \quad \text{and} \quad 
(e^{G} \bu )_m = (\bbeta \ast \bu)_m = \sum_{k=0}^{\infty}\beta_{k} u_{m-k},
\end{align} where $\alpha_k$ and $\beta_k$ are the Fourier modes of $e^{-h}$ and $e^{h}$ in \eqref{ehEq}, and $\balpha,\bbeta$ as in \eqref{balpha_seq}, respectively, \eqref{bbeta_seq}.
Note that $e^{\pm G} $ can also be written in terms of left translation operator as
\begin{align}\label{eGop_leftshift}
 e^{-G} \bu = \sum_{k=0}^{\infty}\alpha_{k}L^{k} \bu , \quad \text{and} \quad e^{G} \bu = \sum_{k=0}^{\infty}\beta_{k}L^{k} \bu,
\end{align} where $L^{k}=\underbrace{L\circ \cdots \circ L}_{k}$ is the $k$-th composition of left translation operator.
It is important to note that the operators $e^{\pm G}$ commute with the left translation, $ [e^{\pm G}, L]=0$.

Different from \cite{sadiqScherzerTamasan}, in this work we carry out the analysis in the Sobolev spaces   $l^{2,p}(\BN;H^q(\OM))$ with the respective norm $\lnorm{\cdot}_{p,q}$ given by 
 \begin{align}\label{spaces}
 l^{2,p}(\BN;H^q(\OM)) &: = \left \{\bu:\; \lnorm{\bu}_{p,q}^2 := \sum_{j=0}^{\INF} (1+j^2)^p \lnorm{ u_{-j} }^{2}_{ H^q(\OM)} < \INF \right \}.
\end{align}In Proposition \ref{eGprop} below we revisit the mapping properties of $e^{\pm G}$ relative to these new spaces.   

Throughout, in the notation of the norms, the first index  $p\in \{0,\frac{1}{2},1\}$  refers to the smoothness in the angular variable (expressed as decay in the Fourier coefficient), while the second index $q\in \{0,1\}$ shows the smoothness in the spatial variable. The most often occurring is the space  $l^2(\BN;L^2(\OM))$, when $p=q=0$. To simplify notation, in this case we drop the double zero subindexes, 
\begin{align*}
 \lnorm{\bu }^2:= \lnorm{\bu}_{0,0}^2=\sum_{j=0}^{\INF} \lnorm{u_{-j}}^{2}_{ L^2(\OM)}.
\end{align*}

The traces on the boundary $\Gam$ of functions in $l^{2,p}(\BN;H^1(\OM))$ are in $ l^{2,p}(\BN;H^{\frac{1}{2}}(\Gam))$, endowed with the norm 
\begin{align} \label{phalfonGamma}
\lnorm{\bg}_{p,\frac{1}{2}}^2 := \sum_{j=0}^{\INF} (1+j^2)^p \lnorm{ g_{-j} }^{2}_{ H^{1/2}(\Gam)}.
\end{align}

Since $\Gamma$ is the unit circle, the $H^{1/2}(\Gam)$-norm can be defined in the Fourier domain as follows. For each integer $j \geq 0$, if we consider the Fourier expansion of the trace $\ds u_{-j} \lvert_{\Gam}$, 
\begin{align*}
\left. u_{-j} \right|_{\Gamma}( e^{i\beta}) = \sum_{k= -\infty}^\infty u_{-j,k}e^{ik\beta}, \quad \text{for} \quad e^{i\beta}\in\Gam,
\end{align*} 
then
\begin{align}\label{onehalfonGamma}
\lnorm{ u_{-j} }^2_{ H^{1/2}(\Gam)}=\sum_{k=-\infty}^\infty(1+k^2)^\frac{1}{2}|u_{-j,k}|^2.
\end{align}
In view of \eqref{onehalfonGamma}, if $\bg\in l^{2,\frac{1}{2}}(\BN;H^{\frac{1}{2}}(\Gam))$, then
\begin{align}\label{unorm1}
\lnorm{\bg}_{\frac{1}{2},\frac{1}{2}}^2  =  \sum_{j=0}^{\INF} \sum_{n=-\INF}^{\INF}  (1+j^2)^\frac{1}{2}(1+n^2)^\frac{1}{2} \lvert g_{-j,n} \rvert^{2}.
\end{align}

In the estimates we need the following variant of the Poincar\' e inequality obtained by component-wise summation: If $\bu \in l^{2}(\BN;H^1(\OM))$, then  
\begin{align}\label{poincare}
\lnorm{\bu}^2  \leq \mu \left( \lnorm{\del \bu}^2 + \lnorm{\bu \lvert_{\Gam} }_{0,\frac{1}{2}}^2 \right),
\end{align} where $\mu$ is a constant depending only on $\OM$; for the unit disc $\mu=2$.

Consider the Banach space 
\begin{align} \label{lone1infdefn}
 l^{1,1}_{\INF}(\ol \OM): = \left \{ \balpha := \langle \alpha_{0}, \alpha_{1}, \alpha_{2}, ... , \rangle: \lnorm{\balpha}_{l^{1,1}_{\INF}(\ol \OM)}:= \sup_{z \in \ol \OM }\sum_{j=1}^{\INF}  j \lvert \alpha_{j}(z) \rvert < \INF \right \}.
\end{align} 
\begin{prop}\label{eGprop}
Let $a\in C^{2,s}(\ol \OM)$, $s>1/2$. Then $\balpha , \del \balpha, \bbeta , \del \bbeta \in l^{1,1}_{\INF} (\ol \OM)$, and 
for any  $p\in \{0,\frac{1}{2},1\}$, $q\in \{0,1\}$, the operators
\begin{align}\label{eGmaps}
 &e^{\pm G}:l^{2,p}(\BN;H^q(\OM))\to l^{2,p}(\BN;H^q(\OM)) 
\end{align} are bounded, and satisfy the following estimates
\begin{align}\label{eG_norm0}
 \lnorm{e^{-G}\bu}   &\leq   \lnorm{\balpha}_{l^{1,1}_{\INF}(\ol \OM)} \lnorm{\bu},  \\ \label{eG_norm_onezero}
 \lnorm{e^{-G}\bu}_{1,0} &\leq 2 \lnorm{\balpha}_{l^{1,1}_{\INF}(\ol \OM)} \lnorm{\bu}_{1,0},  \\ \label{eG_norm_zeroone}
 \lnorm{e^{-G}\bu}_{0,1} & \leq  \left( \lnorm{\balpha}_{l^{1,1}_{\INF}(\ol \OM)} + \lnorm{\del \balpha}_{l^{1,1}_{\INF}(\ol \OM)} \right) \lnorm{\bu}_{0,1}, \\ \label{eG_norm_oneone}
 \lnorm{e^{-G}\bu}_{1,1} & \leq 2\left( \lnorm{\balpha}_{l^{1,1}_{\INF}(\ol \OM)} + \lnorm{\del \balpha}_{l^{1,1}_{\INF}(\ol \OM)} \right) \lnorm{\bu}_{1,1}.
\end{align} The same estimate works for $e^{G}\bu$ with $\balpha$ replaced by $\bbeta$.
\end{prop}
The proof of the Proposition \ref{eGprop} can be found in the Appendix.

We remark that cases $(p=0=q)$, $(p=1, q=0)$, and $(p=0, q=1)$ in Proposition \ref{eGprop}  hold for $a\in C^{1,s}(\ol \OM)$, $s>1/2$, and these are the only properties needed for the lemma below.
\begin{lemma}\label{beltrami_reduction}
Let $a\in C^{1,s}(\ol \OM)$, $s>1/2$, and $e^{\pm G}$ as defined in \eqref{eGop}. 

(i) If $\bu\in l^{2}(\BN;H^1(\OM))$ solves $\ds \dba \bu +L^2 \del \bu+ aL\bu = 0$, then  $\ds \bv= e^{-G} \bu \in l^{2}(\BN;H^1(\OM))$ solves $\dba \bv + L^2\del \bv =0$.

(ii) Conversely, if $\bv\in l^{2}(\BN;H^1(\OM))$ solves $\dba \bv + L^2\del \bv =0$, then $\ds \bu= e^{G} \bv \in l^{2}(\BN;H^1(\OM))$ solves $\ds \dba \bu +L^2 \del \bu+ aL\bu = 0$.
\end{lemma}
\begin{proof}
(i) Let $\ds \bv = e^{-G} \bu = \sum_{k=0}^{\infty}\alpha_{k}L^{k} \bu$.  Since $\bu\in l^{2}(\BN;H^1(\OM))$, then from Proposition \ref{eGprop},  $\ds \bv \in l^{2}(\BN;H^1(\OM))$.
Then $\bv$ solves
\begin{align*}
\dba \bv + L^2\del \bv&= \sum_{k=0}^{\infty} \dba \alpha_{k}L^{k} \bu + \sum_{k=0}^{\infty}\alpha_{k}L^{k}  \dba \bu + \sum_{k=0}^{\infty} \del \alpha_{k}L^{k+2} \bu + \sum_{k=0}^{\infty}\alpha_{k}L^{k+2}  \del \bu \\
&=  \dba \alpha_0 \bu +\dba \alpha_1 L\bu + \sum_{k=0}^{\infty} \left(\dba \alpha_{k+2} + \del \alpha_{k} \right) L^{k+2} \bu + \sum_{k=0}^{\infty}\alpha_{k} L^{k} \left(\dba \bu +L^2 \del \bu \right) \\
&=   \dba \alpha_0 \bu + \dba \alpha_1 L\bu + \sum_{k=0}^{\infty} \left(\dba \alpha_{k+2} + \del \alpha_{k} \right) L^{k+2} \bu + \sum_{k=0}^{\infty}\alpha_{k} L^{k} \left(-a L \bu \right) \\
&=   \dba \alpha_0 \bu + \left(\dba \alpha_1 -a \alpha_0 \right)L\bu + \sum_{k=0}^{\infty} \left(\dba \alpha_{k+2} + \del \alpha_{k} - a \alpha_{k+1} \right) L^{k+2} \bu =0,
\end{align*} where in the last equality we have used \eqref{alphazero_one} and \eqref{alphak}.

An analogue calculation using the properties in Lemma \ref{hproperties} (iv) shows the converse.

\end{proof}


\section{Source reconstruction for scattering of polynomial type}\label{Sec:Sourcerec_poly}

This section contains the basic idea of reconstruction in the special case of scattering kernel of polynomial type, 
\begin{align}\label{particular_kM}
 k(z,\cos \tta)  = \sum_{n=0}^{M} k_{-n}(z) \cos (n \tta), 
\end{align} for some fixed integer $M\geq1$. 
Recall that since $k(z, \cos \tta)$ is both real valued and even in $\tta$,  $k_n(z)$ are real valued and $k_n(z)=k_{-n}(z)$, $0 \leq n \leq M$.
We stress here that no smallness assumption on $k_{0}, k_{-1}, k_{-2}, \cdots, k_{-M}$ is assumed. 
Let $u^{(M)}$ be the solution of \eqref{TransportScatEq1} with $k$ as in \eqref{particular_kM} and $\bu^{(M)}$ denote the sequence valued map 
\begin{align}\label{buM}
 \OM \ni z\mapsto  \bu^{(M)}(z): = \langle u^{(M)}_{0}(z), u^{(M)}_{-1}(z),u^{(M)}_{-2}(z), \cdots, u^{(M)}_{-M}(z), u^{(M)}_{-M-1}(z), \cdots  \rangle.
\end{align}
Let also $K^{(M)}$ denote the corresponding Fourier multiplier operator
\begin{align}\label{KMmultiplier}
K^{(M)}\bu^{(M)}(z)=(k_0(z) u_0^{(M)}(z),k_{-1}(z)u_{-1}^{(M)}(z),\cdots, k_{-M}(z)u_{-M}^{(M)}(z), 0,0,\cdots).
\end{align}
The transport equation \eqref{TransportScatEq1} reduces to the system
\begin{align}\label{freeMeq}
\overline{\del u^{(M)}_{-1}(z)} +\del u^{(M)}_{-1}(z) + a(z)u^{(M)}_{0}(z) &= k_{0}(z)u^{(M)}_{0}(z)+f^{(M)}(z), \\ \label{fMinsys}
\dba u^{(M)}_{-n}(z) +\del u^{(M)}_{-n-2}(z) + a(z)u^{(M)}_{-n-1}(z) &= k_{-n-1}(z)u^{(M)}_{-n-1}(z),\quad 0 \leq n\leq M-1, \\ \label{infMsys}
\dba u^{(M)}_{-n}(z) +\del u^{(M)}_{-n-2}(z) + a(z)u^{(M)}_{-n-1}(z) &=0,\qquad n\geq M.
\end{align}
In sequence valued notation, the system \eqref{fMinsys} and \eqref{infMsys} rewrites:
\begin{align}\label{beltramiA_polyVM}
\dba\bu^{(M)} +L^2 \del\bu^{(M)}+ a(z)L\bu^{(M)} = L K^{(M)}\bu^{(M)},
\end{align} where $K^{(M)}$ as in \eqref{KMmultiplier}. 

Since $f\in H^1(\Omega)$, the solution $u\in H^1(\OM\times\sph)$, and consequently $\bu^{(M)} \in l^{2,1}(\BN;H^1(\OM))$. We note that in our method we only use $\bu^{(M)} \in l^{2,\frac{1}{2}}(\BN;H^1(\OM))$, indicating that it may apply to rougher sources.

Let the transformation $\bv^{(M)}= e^{-G}L^M\bu^{(M)}$, then by Proposition \ref{eGprop}, $\bv^{(M)} \in l^{2,\frac{1}{2}}(\BN;H^1(\OM))$, and $\bv^{(M)}$ is $L^2$-analytic:
\begin{equation}\label{beltrami_VM}
\dba \bv^{(M)} +L^2 \del \bv^{(M)} = 0.
\end{equation}
The trace of the boundary $\bv^{(M)} \lvert_{\Gam}$ is determined by the trace of $\bu^{(M)}\lvert_{\Gam} = \bg = (g_0, g_{-1},g_{-2},...)\in l^{2,\frac{1}{2}}(\BN;H^{1/2}(\Gam)) $,  by
\begin{align}\label{w_boundary}
 \bv^{(M)} \lvert_{\Gam} = e^{-G} L^M \bu^{(M)} \lvert_{\Gam}=e^{-G} L^M \bg.
\end{align} By Proposition \ref{eGprop}, $\bv^{(M)} \lvert_{\Gam}\in l^{2,\frac{1}{2}}(\BN;H^{1/2}(\Gam)) $.

The Bukhgeim-Cauchy integral formula \eqref{CauchyBukhgeimformula} extends $\bv^{(M)}$ from $\Gam$ to $\OM$ as $L^2$-analytic map. 
From the uniqueness of an $L^2$-analytic map with a given trace, we recovered for $n \geq 0,$
\begin{align}\label{w_inside}
v^{(M)}_{-n}(z)&=\frac{1}{2\pi i}\int_\Gam\frac{v^{(M)}_{-n}(\zeta)}{\zeta-z}d\zeta 
+ \frac{1}{2\pi i} \int_\Gam \left\{
\frac{d\zeta}{\zeta-z} - \frac{d\ol{\zeta}}{\ol{\zeta}-\ol{z}}
\right\}\sum_{j=1}^\infty v^{(M)}_{-n-2j}(\zeta)\left(\frac{\ol{\zeta-z}}{\zeta-z}\right)^j,z\in\OM.
\end{align}
Thus $\bv^{(M)}=\langle v^{(M)}_{0}, v^{(M)}_{-1}, v^{(M)}_{-2}, ... \rangle$ is recovered in $l^{2,\frac{1}{2}}(\BN;H^1(\OM))$.

We recover $L^M \bu^{(M)} = \langle u^{(M)}_{-M}, u^{(M)}_{-M-1}, u^{(M)}_{-M-2}, ... \rangle$ in $\OM$ by using the convolution formula \eqref{eGop}
\begin{align}\label{LMv_interms_w}
 u^{(M)}_{-n-M}(z) = \sum_{k=0}^{\INF} \beta_k(z) v^{(M)}_{-n-k}(z), \quad z\in \OM, \; n \geq 0,
\end{align} where $\beta_k$'s as in \eqref{ehEq}.
In particular we recovered $u^{(M)}_{-M-1}, u^{(M)}_{-M}\in H^1(\OM)$.

By applying $4 \del $ to \eqref{fMinsys}, the mode $u^{(M)}_{-M+1}$ is then the solution to the Dirichlet problem for the  Poisson equation 
\begin{subequations} \label{Transport_VM}
\begin{align} \label{Poisson_VM}
\Laplace u^{(M)}_{-M+1} &= -4\del^2 u^{(M)}_{-M-1} -4\del \left[ (a- k_{-M}) u^{(M)}_{-M} \right],\\ 
\label{VM_Gam}  u^{(M)}_{-M+1} \lvert_{\Gam} &= g_{-M+1},
\end{align} where the right hand side of \eqref{Poisson_VM} is known.
\end{subequations}

Since by construction $u^{(M)}_{-M}, u^{(M)}_{-M-1}\in H^1(\OM)$, we have
\begin{align*}
\lnorm{ \del^2 u^{(M)}_{-M-1} + \del \left((a- k_{-M}) u^{(M)}_{-M} \right) }_{H^{-1}(\OM)}^2 &\leq \lnorm{ \del u^{(M)}_{-M-1}}_{L^2(\Omega)}^2+ \lnorm{\left((a- k_{-M}) u^{(M)}_{-M} \right) }_{L^2(\Omega)}^2\\
&\leq \lnorm{ \del u^{(M)}_{-M-1}}_{L^2(\Omega)}^2+  \lnorm{a- k_{-M}}_{L^\infty(\Omega)}^2 \lnorm{ u^{(M)}_{-M} }_{L^2(\Omega)}^2 \\
&\leq \lnorm{ u^{(M)}_{-M-1}}_{H^1(\Omega)}^2 + \lnorm{a- k_{-M}}_{L^\infty(\Omega)}^2 \lnorm{ u^{(M)}_{-M} }_{H^1(\Omega)}^2.
\end{align*}
Since $g_{-M+1} \in H^{1/2}(\Gam)$, the solution $u^{(M)}_{-M+1} \in H^{1}(\OM)$ and 
\begin{align}\label{uM_est}
 \lnorm{ u^{(M)}_{-M+1} }_{H^1(\OM)}^2 \leq C \left( \lnorm{u^{(M)}_{-M-1}}_{H^{1}(\OM)}^2 +  \lnorm{u^{(M)}_{-M}}_{H^1(\Omega)}^2 + \lnorm{g_{-M+1}}_{H^{1/2}(\Gam)}^2      \right),
\end{align} where the constant $C$ depends only on $\OM$ and 
$\ds \max \left\{ 1, \max_{0\leq j \leq M}    \lnorm{a- k_{-j}}_{L^\infty(\Omega)}^2  \right\}$.
Successively all the other modes $u^{(M)}_{-M+j}$ for $j = 2, \cdots, M$ are computed by solving the corresponding Dirichlet problem for the  Poisson equation.
To account for the successive accumulation of error we note the following result which can be proven by induction.
\begin{lemma}\label{rec_result}
Let $\{a_n\}$ and $\{b_n\}$  be sequences of nonnegative numbers, such that 
\begin{align*}
a_{n+2} \leq c \left(a_{n+1} +a_{n} +b_{n+2} \right), \quad n \geq 0,
\end{align*} where $c>0$ is a constant, then 
\begin{align*}
a_{n+2} \leq (1+c)^{n+2} \left(a_{1} +a_{0}  + \sum_{k=0}^{n} b_{k+2} \right), \quad n \geq 0.
\end{align*} 
\end{lemma}

We applying Lemma \ref{rec_result} to \eqref{uM_est} and estimate 
\begin{align}\label{uMminus1_est}
 \lnorm{u^{(M)}_{-1}}_{H^{1}(\OM)}^2 \leq (1+C)^{M-1} \left( \lnorm{u^{(M)}_{-M-1}}_{H^{1}(\OM)}^2 +   \lnorm{ u^{(M)}_{-M} }_{H^1(\Omega)}^2 + \sum_{j=1}^{M-1}   \lnorm{g_{-M+j}}_{H^{1/2}(\Gam)}^2      \right),
\end{align} 
and 
\begin{align}\label{uMzero_est}
 \lnorm{u^{(M)}_{0}}_{H^{1}(\OM)}^2 \leq (1+C)^{M} \left( \lnorm{u^{(M)}_{-M-1}}_{H^{1}(\OM)}^2 +   \lnorm{ u^{(M)}_{-M} }_{H^1(\Omega)}^2 + \sum_{j=1}^{M} \lnorm{g_{-M+j}}_{H^{1/2}(\Gam)}^2      \right).
\end{align} 

The source $f^{(M)}$ is computed by 
\begin{align} \label{fsource_Scatpoly}
 f^{(M)}(z) = 2 \re \left(\del u^{(M)}_{-1}(z)\right) + \left( a(z)-k_0(z) \right) u^{(M)}_0(z),
\end{align} and we estimate 
\begin{align}\nonumber
 \lnorm{f^{(M)}}_{L^{2}(\OM)}^2 &\leq 2 \lnorm{ u^{(M)}_{-1} }_{H^{1}(\OM)}^2 + C \lnorm{ u^{(M)}_{0} }_{H^{1}(\OM)}^2\\ \label{fMest}
 &\leq 2(1+C)^{M+1} \left( \lnorm{ u^{(M)}_{-M-1} }_{H^{1}(\OM)}^2 +   \lnorm{ u^{(M)}_{-M}}_{H^1(\Omega)}^2 + \sum_{j=1}^{M} \lnorm{g_{-M+j}}_{H^{1/2}(\Gam)}^2      \right).
\end{align}

This method is implemented in the numerical experiments in Section \ref{Sec:numerics}. Next we analyze the error introduced by truncation.

\section{Gradient estimates of solutions to nonhomogeneous Bukhgeim-Beltrami equation}\label{Sec:Gradientest_BBeq}

When applying the reconstruction method to the data arising from a general scattering kernel $\ds k(z,\cos \tta)  = \sum_{n=0}^{\INF} k_{-n}(z) \cos (n \tta)$ an error is made due to the truncation in the Fourier series of $k$. This error is controlled by the gradient of the solution to the Cauchy problem for the inhomogeneous Bukhgeim-Beltrami equation
\begin{align}\label{general_beltrami}
\dba\bv +L^2 \del\bv = B\bv + \bbf,
\end{align}for some specific $\bbf$ and operator coefficient $B$. Estimates of the gradient for solutions of \eqref{general_beltrami} may be of separate interest, reason for which we treat them here 
independently of the transport problem.

We start with an energy identity (see \cite{tamasan03} for $\bbf=0$), \`{a} la Mukhometov \cite{mukhometov75} or Pestov \cite{sharafutdinov97}.

\begin{theorem}(Energy identity)\label{newidentity}
Let $\bbf \in l^{2,1}(\BN;L^2(\OM))$ and let $B$ be a bounded operator such that  $B:l^2(\BN;L^2(\OM))\to l^{2,1}(\BN;L^2(\OM))$ and  $B:l^2(\BN;H^1(\OM))\to l^{2,1}(\BN;H^1(\OM))$.\\
If $\bv \in l^{2,\frac{1}{2}}(\BN;H^1(\OM))$ is a solution to \eqref{general_beltrami},
then    
\begin{align}\label{identity_source}
\int_{\OM}  \lnorm{\del \bv}_{l_2}^2dx &=
-2 \int_{\OM} \sum_{j=1}^\infty \re \spr{L^{2j}\del \bv}{L^{2j-2}B \bv} dx
+ \int_{\OM} \sum_{j=0}^\infty \lnorm{L^{2j}B \bv}_{l_2}^2 dx   \\ \nonumber 
&\quad -2 \int_{\OM} \sum_{j=1}^\infty \re \spr{L^{2j}\del \bv}{L^{2j-2}\bbf} dx
+2 \int_{\OM} \sum_{j=0}^\infty \re \spr{L^{2j}B \bv}{L^{2j}\bbf} dx \\ \nonumber
&\quad + \int_{\OM} \sum_{j=0}^\infty \lnorm{L^{2j}\bbf}_{l_2}^2 dx + \frac{i}{2} \int_{\Gam} \sum_{j=0}^\infty \spr{L^{2j}\bv}{\del_s L^{2j}\bv}ds.
\end{align}
\end{theorem}
\begin{proof}
Using the Green's identity
$\ds2\int_{\OM} \lnorm{\del\bv}_{l_2}^2dx=2\int_{\OM}\lnorm{\dba\bv}_{l_2}^2dx + i\int_{\Gam} \spr{\bv}{\del_s\bv}ds,$ where $\del_s$ is the tangential derivative at the boundary, it follows that
\begin{align*}
\int_{\OM}\lnorm{\del\bv}_{l_2}^2dx & = \int_{\OM}\lnorm{L^2 \del\bv - B\bv - \bbf}_{l_2}^2dx +\frac{i}{2} \int_{\Gam} \spr{\bv}{\del_s \bv}ds\\
&= \int_{\OM}\lnorm{L^2 \del\bv}_{l_2}^2dx -2 \int_{\OM} \re \spr{L^{2}\del \bv}{ B \bv} dx
+ \int_{\OM}  \lnorm{B \bv}_{l_2}^2 dx \\
&\quad -2 \int_{\OM} \re \spr{L^{2}\del \bv}{\bbf} dx
+2 \int_{\OM} \re \spr{ B \bv}{\bbf} dx
+ \int_{\OM}  \lnorm{\bbf}_{l_2}^2 dx +\frac{i}{2} \int_{\Gam} \spr{\bv}{\del_s \bv}ds.
\end{align*}
For each $n\in\mathbb{N}$,
\begin{align*}
\int_{\OM}\lnorm{L^{2n}\del\bv}_{l_2}^2dx &= \int_{\OM}\lnorm{L^{2n+2} \del\bv}_{l_2}^2dx -2 \int_{\OM} \re \spr{L^{2n+2}\del \bv}{L^{2n} B \bv} dx
+ \int_{\OM}  \lnorm{L^{2n} B \bv}_{l_2}^2 dx \\
&\quad -2 \int_{\OM} \re \spr{L^{2n+2}\del \bv}{L^{2n} \bbf} dx
+2 \int_{\OM} \re \spr{L^{2n} B \bv}{L^{2n} \bbf} dx \\
&\quad + \int_{\OM}  \lnorm{L^{2n} \bbf}_{l_2}^2 dx + \frac{i}{2} \int_{\Gam} \spr{L^{2n}\bv}{\del_s L^{2n}\bv}ds.
\end{align*}
By summing in $n$, and using $\ds \lim_{n\to\infty}\int_{\OM}\lnorm{L^{2n}\del\bv}_{l_2}^2 dx=0$, we conclude the theorem.
\end{proof}
We note the general identity \cite[Lemma 2.1]{sadiqTamasan01}, for a sequence of nonnegative numbers:
\begin{lemma}\label{seqresult}
Let $\{c_n\}$ be a sequence of nonnegative numbers. Then\begin{align*}
(i)\quad \sum_{m=0}^{\INF} \sum_{n=0}^{\INF} c_{m+n} = \sum_{j=0}^{\INF} (1+j)\; c_{j}, \quad 
(ii)\quad \sum_{m=0}^{\INF} \sum_{n=0}^{\INF} c_{m+2n} = \sum_{j=0}^{\INF} \left(1+\floor*{\frac{j}{2}}\right) \; c_{j},
\end{align*} whenever one of the sides in (i) and (ii) is finite. 
\begin{proof}
(i) By changing the index $j=m+n$, for $m\geq 0$, ($j-n \geq 0,$ and $n \leq j$), we get 
\begin{align*}
  \sum_{m=0}^{\INF} \sum_{n=0}^{\INF}c_{m+n} = \sum_{j=0}^{\INF} \sum_{n=0}^{j}c_{j} =   \sum_{j=0}^{\INF} c_{j} \sum_{n=0}^{j} 1
= \sum_{j=0}^{\INF} (1+j) c_j.
\end{align*}
(ii) Similarly, by changing the index $j=m+2n$, for $m\geq 0$, $ \left( j-2n \geq 0, \; \text{and} \; \ds n \leq \floor*{\frac{j}{2}} \right)$, we get 
\begin{align*}
 \sum_{m=0}^{\INF} \sum_{n=0}^{\INF}c_{m+2n} = \sum_{j=0}^{\INF} \sum_{n=0}^{\floor*{\frac{j}{2}}}c_{j} 
= \sum_{j=0}^{\INF} \left(1+\floor*{\frac{j}{2}}\right) \; c_{j}.
\end{align*}
\end{proof}
\end{lemma}

\begin{theorem}[Gradient estimate]\label{new_estimate}
Let $\bbf \in l^{2,1}(\BN;L^2(\OM))$, and $B$ be a smoothing operator such that $B:l^2(\BN;L^2(\OM))\to l^{2,1}(\BN;L^2(\OM))$ and  $B:l^2(\BN;H^1(\OM))\to l^{2,1}(\BN;H^1(\OM))$ are bounded, and let  $C_B >0$ be such that
\begin{subequations} \label{B_opnorm}
\begin{align} \label{B_opnorm10}
\lnorm{B \bv}_{1,0} &\leq C_B \lnorm{\bv}, \qquad  \forall \bv \in l^2(\BN;L^2(\OM)),   \\  \label{B_opnorm11}
 \lnorm{B \bv}_{1,1} &\leq C_B \lnorm{\bv}_{0,1}, \quad \, \forall \bv \in l^2(\BN;H^1(\OM)).
\end{align} 
\end{subequations}
Assume that $B$ is such that
\begin{align}\label{eps_defn1}
 \epsilon := \sqrt{2\mu}C_B < \sqrt{2}-1, 
\end{align} where $\mu$ is the factor in Poincar\' e inequality \eqref{poincare}.

If $\bv \in l^{2,\frac{1}{2}}(\BN;H^1(\OM))$ is a solution to the inhomogeneous Bukgheim-Beltrami equation \eqref{general_beltrami}, then 
\begin{align}\label{newestimate1}
0\leq  \lnorm{\del \bv} \leq \frac{b+\sqrt{b^2+4ac}}{2a},
\end{align}
where
\begin{align*}
 a = 1- 2\epsilon - \epsilon^2>0, \quad
 b = 2 \epsilon \lnorm{\bv \lvert_{\Gam} }_{0,\frac{1}{2}} + 2 \sqrt{2} \lnorm{\bbf}_{1,0}, \quad 
 c = \epsilon^2 \lnorm{\bv \lvert_{\Gam} }_{0,\frac{1}{2}}^{2} +\pi \lnorm{ \bv \lvert_{\Gam} }^{2}_{\frac{1}{2},\frac{1}{2}} + 2\lnorm{\bbf}_{1,0}^2.
\end{align*} 
\end{theorem}

\begin{proof}
We estimate each term on the right hand side of the energy identity \eqref{identity_source}.
For brevity if we denote $B \bv = \langle b_0, b_{-1}, b_{-2}, \cdots \rangle$, then  $ \ds \lnorm{B \bv}_{1,0}^{2}= \sum_{j=0}^\infty  (1+j^2) \int_{\OM}|b_{-j}|^2$. 
\begin{enumerate}[I.]
 \item We estimate the first term in \eqref{identity_source}:
    \begin{align*}
	& \left | \int_{\OM} \sum_{j=1}^\infty \spr{L^{2j}\del \bv}{L^{2j-2}B\bv}  dx \right | = \left | \int_{\OM} \sum_{j=1}^\infty \sum_{k=0}^\infty \del v_{-2j-k} b_{-2j+2-k} \right | \\
	&\qquad \leq \int_{\OM} \sum_{j=1}^\infty  |\del v_{-j}| |j b_{-j+1}| 
	\leq \left( \int_{\OM} \sum_{j=1}^\infty  |\del v_{-j}|^2 \right)^{1/2} \left( \int_{\OM} \sum_{j=0}^\infty  (1+j)^2 |b_{-j}|^2 \right)^{1/2} \\
	&\qquad \leq \sqrt{2} \lnorm{\del \bv} \, \left( \int_{\OM} \sum_{j=0}^\infty  (1+j^2) |b_{-j}|^2 \right)^{1/2}
	\leq  \sqrt{2} \lnorm{\del \bv} \lnorm{B \bv}_{1,0}  \\
	&\qquad \leq \sqrt{2} C_{B}  \lnorm{\del \bv}  \, \lnorm{\bv}  \leq \sqrt{2} C_{B} \lnorm{\del \bv} \sqrt{\mu}  \left( \lnorm{\del \bv} + \lnorm{ \bv\lvert_{\Gam} }_{0,\frac{1}{2}} \right) \\
	&\qquad = \sqrt{2 \mu} C_{B}  \lnorm{\del \bv}^2 + \sqrt{2\mu} C_{B}  \, \lnorm{\del \bv} \, \lnorm{\bv\lvert_{\Gam} }_{0,\frac{1}{2}}, 
    \end{align*} where in the first inequality we use Lemma \ref{seqresult} part (i), in the second inequality we use Cauchy-Schwarz inequality, in the third inequality we use $(1+x)^2 \leq 2(1+x^2) $,
    in the fifth inequality we use $\lnorm{B \bv}_{1,0} \leq C_B\lnorm{\bv}$, and in the next to last inequality we use the Poincar\' e inequality \eqref{poincare}.
 \item We estimate the second term in \eqref{identity_source}:
    \begin{align*}
	\int_{\OM} \sum_{j=1}^\infty & | \spr{L^{2j}\del \bv}{L^{2j-2} \bbf} | =  \int_{\OM} \sum_{j=1}^\infty \sum_{k=0}^\infty |\del v_{-2j-k} f_{-2j+2-k}| \leq  \int_{\OM} \sum_{j=1}^\infty  |\del v_{-j}| |j f_{-j+1}| \\
	&\leq  \left( \int_{\OM} \sum_{j=1}^\infty  |\del v_{-j}|^2 \right)^{1/2} \left( \int_{\OM} \sum_{j=0}^\infty  (1+j)^2 |f_{-j}|^2 \right)^{1/2} 
	 \leq \sqrt{2} \lnorm{\del \bv} \, \lnorm{\bbf}_{1,0}, 
    \end{align*} where in the first inequality we use Lemma \ref{seqresult} part (i), and then Cauchy-Schwarz.
 \item We estimate the third term in \eqref{identity_source}:  
    \begin{align*}
	\int_{\OM} \sum_{j=0}^\infty \lnorm{|L^{2j}B \bv}_{l_2}^2 &=  \int_{\OM} \sum_{j=0}^\infty \sum_{k=0}^\infty | b_{-2j-k}|^2 \leq  \int_{\OM} \sum_{j=0}^\infty (1+j)| b_{-j}|^2 
	\leq \int_{\OM} \sum_{j=0}^\infty  (1+j^2)  |b_{-j}|^2 \\
	&\qquad \leq \lnorm{B \bv}^2_{1,0} \leq C_B^2 \lnorm{\bv}^2 \leq   \mu C_B^2  \lnorm{\del \bv}^2 + \mu C_B^2 \lnorm{\bv \lvert_{\Gam} }_{0,\frac{1}{2}}^2, 
    \end{align*} where in the first inequality we use Lemma \ref{seqresult} part (i), in the next to the last inequality we use $\lnorm{B \bv}_{1,0} \leq C_B \lnorm{\bv}$, while in the last we use the Poincar\' e inequality \eqref{poincare}.
 \item We estimate the fourth term in \eqref{identity_source}: 
    \begin{align*}
	\int_{\OM} \sum_{j=0}^\infty \lnorm{L^{2j}\bbf}_{l_2}^2  &=  \int_{\OM} \sum_{j=0}^\infty \sum_{k=0}^\infty | f_{-2j-k}|^2 \leq  \int_{\OM} \sum_{j=0}^\infty (1+j)| f_{-j}|^2   
	\leq   \int_{\OM} \sum_{j=0}^\infty  (1+j^2) |f_{-j}|^2=\lnorm{\bbf}^2_{1,0},
    \end{align*}  where in the first inequality we have used Lemma \ref{seqresult} part (i).
 \item We estimate the next to the last term in \eqref{identity_source}: 
    \begin{align*}
	2 \int_{\OM} \sum_{j=0}^\infty  | \spr{L^{2j}B \bv}{L^{2j}\bbf} | &\leq 2\int_{\OM} \sum_{j=0}^\infty  \lnorm{L^{2j}B\bv}_{l_2} \, \lnorm{L^{2j}\bbf}_{l_2} 
	\leq  \int_{\OM} \sum_{j=0}^\infty \lnorm{L^{2j}B\bv}_{l_2}^2 +  \int_{\OM} \sum_{j=0}^\infty \lnorm{L^{2j}\bbf}_{l_2}^2 \\
	&\quad \leq  \mu C_B^2  \lnorm{\del \bv}^2 + \mu C_B^2 \lnorm{ \bv \lvert_{\Gam} }_{0,\frac{1}{2}}^2 + \lnorm{\bbf}^2_{1},
    \end{align*}  where in the first inequality we have used Cauchy-Schwarz, in the second inequality we have used the fact $xy \leq \frac{1}{2} (x^2+y^2)$, and in the last inequality we have used estimates from (III) and (IV).
  \item  To estimate the last term in \eqref{identity_source}, we use the fact that $\Gam$ is the unit circle and consider the Fourier expansion of the traces of the modes $v_{-j} \lvert_{\Gam}$. 
  For $e^{i \beta} \in \Gam$, 
  \begin{align*}
      v_{-j} ( e^{i\beta}) = \sum_{k= -\infty}^\infty v_{-j,k}e^{ik\beta}, \qquad \del_{\beta}v_{-j} ( e^{i\beta}) = \sum_{m= -\infty}^\infty (im) v_{-j,m}e^{im\beta}.
  \end{align*} 
  Using the parametrization, the last term in \eqref{identity_source} becomes 
  \begin{align}\nonumber
  \left| \int_{\Gam} \sum_{j=0}^\infty \spr{L^{2j}\bv}{\del_s L^{2j}\bv}ds \right| &=
	\left| \int_{\Gam} \sum_{j=0}^\infty \sum_{k=0}^\infty v_{-2j-k} \ol{\del_s v_{-2j-k}} ds \right| 
	= \left| \int_{\Gam} \sum_{j=0}^\infty  \left(1+\floor*{\frac{j}{2}}\right)   v_{-j} \ol{\del_s v_{-j}} ds \right| \\ \nonumber
 & =  \left| \int_{0}^{2\pi} \sum_{j=0}^\infty  \left(1+\floor*{\frac{j}{2}}\right) \sum_{k= -\infty}^\infty v_{-j,k}e^{ik\beta} \sum_{m= -\infty}^\infty (-im) \ol{v_{-j,m}}e^{-im\beta} d\beta \right| \\ \nonumber
 & =  \left| \sum_{j=0}^\infty  \left(1+\floor*{\frac{j}{2}}\right) \sum_{k= -\infty}^\infty v_{-j,k} \sum_{m= -\infty}^\infty (-im) \;  \ol{v_{-j,m}}  \int_{0}^{2\pi}  e^{i(k-m)\beta}d\beta \right| \\ \nonumber
 &\quad \leq  2 \pi  \sum_{j=0}^\infty   (1+j^2)^{\frac{1}{2}} \sum_{k= -\infty}^\infty |k|\,\left|  v_{-j,k} \ol{v_{-j,k}} \right| \\ \label{boundary_lastterm}
 & \quad \leq 2 \pi \sum_{j=0}^\infty  \sum_{k= -\infty}^\infty (1+j^2)^{\frac{1}{2}} (1+k^2)^{\frac{1}{2}} \left| v_{-j,k}\right|^2 = 2 \pi \lnorm{\bv \lvert_{\Gam}}_{\frac{1}{2},\frac{1}{2}}^2,
\end{align} where in the second equality we have used Lemma \ref{seqresult} part (ii), in the first inequality we have used the fact $\left(1+\floor*{\frac{j}{2}}\right) \leq  (1+j^2)^{1/2} $,  and in the last equality we have used the definition of the norm \eqref{unorm1}.    
\end{enumerate}
Using the above estimates (I)-(VI) for the expressions in \eqref{identity_source}, we have proved for 
$\ds \tau =   \left( \int_{\OM}  \lnorm{\del \bv}_{l_2}^2 \right)^{1/2} $, that $a \tau^2 - b\tau - c \leq 0$.
Assumption on $\epsilon$ as in \eqref{eps_defn1} yield $a>0$ and we have the estimate \eqref{newestimate1}.
\end{proof}
For the case when $B=0$ we obtain the immediate corollary.
\begin{cor}\label{gradEstCor}
Let $\bbf \in l^{2,1}(\BN;L^2(\OM))$. If $\bv \in l^{2,\frac{1}{2}}(\BN;H^1(\OM))$ solves 
\begin{align}\label{BukhBeltsimple}
\dba\bv +L^2 \del\bv = \bbf,
\end{align}
then 
\begin{align}\label{estimateSimple}
 \lnorm{\del \bv}^2 \leq 12 \lnorm{\bbf}_{1,0}^2+2\pi \lnorm{\bv \lvert_{\Gam}}^2_{\frac{1}{2},\frac{1}{2}}.
\end{align}
\end{cor}
\begin{proof}
This is the case $\epsilon=0$ and $a=1$ in \eqref{newestimate1}. We also use
$\ds\left(\frac{b+\sqrt{b^2+4c}}{2}\right)^2\leq b^2+2c.$
\end{proof}
\section{Smoothing due to scattering}\label{Sec:SmoothingK}

In this section we explicit the smoothing properties of the Fourier multiplier operator $K$ in \eqref{multiplier} as determined by the appropriate decay of the Fourier coefficients of the scattering kernel $k_n(z)=\frac{1}{2\pi} \int_{-\pi}^{\pi} k(z,\cos \theta ) e^{-i n\theta}d\theta$, $n\in\mathbb{Z}$.
The gain of $1/2$  smoothness in the angular variable (see \eqref{LNKvphalf} below) has been shown before by different methods in \cite{mokhtar}, and in a more general case than considered here in \cite{stefanovUhlmann08}.
\begin{lemma}(Smoothing due to scattering)\label{key} 
Let $M \geq1$ be a positive integer and $K$ be the Fourier multiplier in \eqref{multiplier}. 
Assume that $k$ is such that its Fourier coefficients starting from index $M$ onward satisfy 
\begin{align}\label{gamma_kdecay}
\gamma := \sup_{j \geq M} (1+j)^{p} \max \left\{ \lnorm{k_{-j}}_{\INF}, \lnorm{\nabla_x k_{-j}}_{\INF} \right\} < \INF, \qquad \text{for} \; p > 1/2.  
\end{align}
(i) If $\bv \in l^2(\BN;L^2(\OM))$, then   
\begin{align}\label{LnEsqr} 
\sum_{n=0}^\infty \lnorm{L^{n+M} K \bv}^2  \leq \frac{\gamma^2}{(M+1)^{2p-1}}  \lnorm{L^M\bv}^2.
\end{align}
(ii) $K:l^2(\BN;H^1(\OM))\to l^{2,\frac{1}{2}}(\BN;H^1(\OM))$ is bounded. More precisely,
\begin{align}\label{LNKvphalf} 
\lnorm{L^{M} K  \bv }_{\frac{1}{2},1 } \leq \frac{\sqrt{2}\gamma}{(M+1)^{p-1/2}} \lnorm{ L^M \bv}_{0,1}, \qquad  \forall \bv \in l^2(\BN;H^1(\OM)).
\end{align}
(iii) Moreover, if \eqref{gamma_kdecay} holds for $p \geq 1$, then
$\ds K:l^2(\BN;L^2(\OM))\to l^{2,1}(\BN;L^2(\OM))$ and\\ $\ds K:l^{2}(\BN;H^1(\OM)) \to l^{2,1}(\BN;H^1(\OM))
$ are bounded, and  
\begin{align}\label{LNp1} 
\lnorm{L^{M} K \bv}_{1,0}  &\leq \frac{\gamma}{(M+1)^{p -1}}  \lnorm{L^M\bv}, \qquad  \forall \bv \in l^2(\BN;L^2(\OM)), \\ \label{LNKvp1} 
\lnorm{L^{M} K  \bv}_{1,1} &\leq \frac{\sqrt{2}\gamma}{(M+1)^{p-1}} \lnorm{L^M \bv}_{0,1}, \qquad  \forall \bv \in l^2(\BN;H^1(\OM)).
\end{align} 
\end{lemma}
\begin{proof}
(i) Let $\bv \in l^{2}(\BN;L^2(\OM))$. Then 
\begin{align*} 
\sum_{n=0}^\infty \lnorm{L^{n+M} K \bv}^2 =& \sum_{n=0}^\infty  \sum_{m=0}^\infty  \int_{\OM}  | k_{-n-m-M}v_{-n-m-M}|^2 = \sum_{j=0}^\infty (1+j) \int_{\OM} | k_{-j-M}v_{-j-M}|^2   \\
&\leq \gamma^2 \sum_{j=0}^\infty \frac{1+j}{(1+j+M)^{2p}} \int_{\OM} | v_{-j-M}|^2
 \leq \gamma^2 \sum_{j=0}^\infty \frac{1}{(1+j+M)^{2p-1}} \int_{\OM} | v_{-j-M}|^2 \\
&\leq  \frac{\gamma^2}{(M+1)^{2p-1}} \sum_{j=0}^\infty \int_{\OM} | v_{-j-M}|^2 
\leq \frac{\gamma^2}{(M+1)^{2p-1}} \lnorm{L^M\bv}^2,
\end{align*}where in the second equality we have used Lemma \ref{seqresult} part (i), and in the first inequality we have used \eqref{gamma_kdecay}. 

To prove (ii), let $\bv \in l^2(\BN;H^1(\OM))$.  Then 
\begin{align} \nonumber
 &\lnorm{L^{M} K \bv}^2_{\frac{1}{2},1 } = \sum_{j=0}^\infty (1+j^2)^{\frac{1}{2}} \left( \int_{\OM} |k_{-j-M}v_{-j-M} |^2 +  \int_{\OM} |\del(k_{-j-M}v_{-j-M}) |^2 \right )\\ \label{Kfirst(i)}
 &\quad \leq \sum_{j=0}^\infty (1+j^2)^{\frac{1}{2}} \left( \int_{\OM} |k_{-j-M}v_{-j-M} |^2 + \int_{\OM} |k_{-j-M}|^2|\del(v_{-j-M}) |^2 + \int_{\OM} |v_{-j-M}|^2|\del k_{-j-M} |^2 \right).
\end{align}
We estimate the first term in \eqref{Kfirst(i)}, 
\begin{align} \nonumber
\sum_{j=0}^\infty &(1+j^2)^{\frac{1}{2}}  \int_{\OM} |k_{-j-M}v_{-j-M} |^2 \leq \gamma^2 \sum_{j=0}^\infty \frac{(1+j^2)^{\frac{1}{2}}}{(1+j+M)^{2p}} \int_{\OM} | v_{-j-M}|^2 \\ \nonumber
&\qquad \leq \gamma^2 \sum_{j=0}^\infty \frac{1}{(1+j+M)^{2p-1}} \int_{\OM}  | v_{-j-M}|^2 \leq  \frac{\gamma^2}{(M+1)^{2p-1}} \sum_{j=0}^\infty  \int_{\OM} | v_{-j-M}|^2 \\ \label{first_term}
&\qquad \leq \frac{\gamma^2}{(M+1)^{2p-1}} \lnorm{ L^M\bv}^2,
\end{align} where in the first inequality we have used decay of $k_j$'s. \\
Similarly, following the same proof as of the first term \eqref{first_term}, the term 
\begin{align*}
 \sum_{j=0}^\infty (1+j^2)^{\frac{1}{2}}  \int_{\OM} |k_{-j-M}|^2|\del(v_{-j-M}) |^2 \leq \frac{\gamma^2}{(M+1)^{2p-1}} \lnorm{ L^M \del \bv}^2, 
\end{align*} and the last term 
\begin{align*} \sum_{j=0}^\infty &(1+j^2)^{\frac{1}{2}}  \int_{\OM} |v_{-j-M}|^2|\del(k_{-j-M}) |^2 \leq \frac{\gamma^2}{(M+1)^{2p-1}} \lnorm{ L^M\bv}^2.
\end{align*}
Thus the expression in \eqref{Kfirst(i)} becomes
\begin{align*} 
 \lnorm{L^{M} K \bv}^2_{\frac{1}{2},1 } &\leq \frac{2\gamma^2}{(M+1)^{2p-1}} \lnorm{ L^M\bv}^2 + \frac{\gamma^2}{(M+1)^{2p-1}} \lnorm{ L^M \del \bv }^2 \\
 &\leq \frac{2\gamma^2}{(M+1)^{2p-1}} \left( \lnorm{ L^M\bv}^2 + \lnorm{ L^M \del \bv}^2   \right) = \frac{2\gamma^2}{(M+1)^{2p-1}} \lnorm{ L^M \bv}^2_{0,1}.
\end{align*}

To prove (iii), assume that $k$ is such that \eqref{gamma_kdecay}  holds for $p \geq 1$.
Let $\bv \in l^2(\BN;L^2(\OM))$, then
\begin{align*} 
\lnorm{L^{M} K \bv}^2_{1,0} &= \sum_{j=0}^\infty (1+j^2) \int_{\OM} |k_{-j-M}v_{-j-M} |^2 \\
&  \leq \gamma^2 \sum_{j=0}^\infty \frac{(1+j^2)}{(1+j+M)^{2p}} \int_{\OM} | v_{-j-M}|^2 \leq \gamma^2 \sum_{j=0}^\infty \frac{1}{(1+j+M)^{2p-2}} \int_{\OM}  | v_{-j-M}|^2 \\
& \leq  \frac{\gamma^2}{(M+1)^{2p-2}} \sum_{j=0}^\infty  \int_{\OM} | v_{-j-M}|^2  \leq \frac{\gamma^2}{(M+1)^{2p-2}} \lnorm{ L^M\bv}^2,
\end{align*} where in the first inequality we have used decay of $k_j$'s. 

To prove the last part, let $\bv \in l^{2}(\BN;H^1(\OM))$.  Then 
\begin{align*}
 &\lnorm{L^{M} K  \bv}_{1,1} = \sum_{j=0}^\infty (1+j^2) \left( \int_{\OM} |k_{-j-M}v_{-j-M} |^2 +  \int_{\OM} |\del(k_{-j-M}v_{-j-M}) |^2 \right ).
\end{align*}
Following the similar proof of (ii), $\left(  \text{with} \; (1+j^2)\; \text{instead of} \; (1+j^2)^{1/2} \right)$, we have 
\begin{align*} 
 &\lnorm{L^{M} K  \bv}^2_{1,1} \leq \frac{2\gamma^2}{(M+1)^{2p-2}} \left( \lnorm{ L^M\bv}^2 + \lnorm{ L^M \del \bv}^2   \right) = \frac{2\gamma^2}{(M+1)^{2p-2}} \lnorm{L^M \bv}^2_{0,1}.
\end{align*}
\end{proof}
\section{Error estimates }\label{Sec:GenScat_Map}

Recall the sequence valued maps 
$\ds\bu = \langle u_{0}, u_{-1},u_{-2},... \rangle$ and $\ds \bu^{(M)} = \langle u^{(M)}_{0}, u^{(M)}_{-1},u^{(M)}_{-2},... \rangle$ solutions of $\ds \dba\bu +L^2 \del\bu+ aL\bu = L K\bu$, respectively of $\dba\bu^{(M)} +L^2 \del\bu^{(M)}+ aL\bu^{(M)} = L K^{(M)}\bu^{(M)}$, where $K$ is the multiplier operator in  \eqref{multiplier}, and $K^{(M)}$ is its truncated version in  \eqref{KMmultiplier} corresponding to the $M$-th order polynomial approximation of $k$.

The error we make in the source reconstruction is controlled by the sequence valued map 
\begin{align}\label{q_defn}
\bq^{(M)}:=e^{-G}(\bu-\bu^{(M)}),
\end{align} 
which solves
\begin{align}\label{bqM_Beltrami}
\dba\bq^{(M)} +L^2 \del\bq^{(M)} =  e^{-G}LK\bu - e^{-G}LK^{(M)}\bu^{(M)}.
\end{align}

Recall that the integrating operators $e^{\pm G}$ commute with the left translation, $ [e^{\pm G}, L]=0$. From the equation \eqref{KMmultiplier} is easy to see that the translated sequence
\begin{align}\label{LMq_defn}
L^M\bq^{(M)}=e^{-G}L^M(\bu-\bu^{(M)})
\end{align}solves
\begin{equation}\label{LMbqM_Beltrami}
\dba L^M\bq^{(M)} +L^2 \del L^M\bq^{(M)} =  e^{-G}L^{M+1}K\bu.
\end{equation}

For $a\in C^{2,s}(\ol \OM)$, $s>1/2$, let $\balpha, \bbeta$ be as in Proposition \ref{eGprop}. For brevity, let us define
\begin{align}\label{sigma_defn}
 \sigma = \max\{ \lnorm{\balpha}_{l^{1,1}_{\INF}(\ol \OM)} , \lnorm{\del \balpha}_{l^{1,1}_{\INF}(\ol \OM)} , \lnorm{\bbeta}_{l^{1,1}_{\INF}(\ol \OM)}, \lnorm{\del \bbeta}_{l^{1,1}_{\INF}(\ol \OM)} \}.
\end{align}

Under sufficient regularity on $k$, its Fourier coefficients satisfy \eqref{gamma_kdecay}  with $p=1$.
\begin{lemma}\label{lem_propk}
If $k\in C^{1,s}_{per}([-1,1]; Lip(\Omega) )$, $s > 1/2$, then  
\begin{align}\label{p=1gamma}
\gamma_M=\sup_{j \geq M+1} (1+j) \max \left\{ \lnorm{ k_{-j}}_{\INF},\lnorm{ \nabla_x k_{-j}}_{\INF} \right\} < \INF, \qquad \text{for any}\; M\geq 1.
\end{align}
\end{lemma}
\begin{proof}
 The proof follows directly by Bernstein's lemma \cite[Chap. I, Theorem 6.3]{katznelson}.
\end{proof}

As a consequence we obtain the following a priori error estimate for the source reconstructed by the method in Section \ref{Sec:Sourcerec_poly}.
\begin{theorem}(Error estimate)\label{error_estimate}
Let $M \geq 1$ be a positive integer, $a\in C^{2,s}(\ol \OM)$, $k\in C^{1,s}_{per}([-1,1]; Lip(\Omega) )$, for $s>1/2$, be known, and $f\in H^1(\OM)$ and $u\in H^1(\OM\times\sph)$ be unknown functions satisfying \eqref{TransportScatEq1}.
Let $\bu$ be the unknown sequence of nonpositive Fourier modes of $u$ as in \eqref{boldu}, and $f^{(M)}$ in \eqref{fsource_Scatpoly} be the reconstructed source using the  noisy boundary data $\widetilde{\bg}\lvert_{\Gam} \in l^{2,\frac{1}{2}}(\BN;H^{1/2}(\Gam))$. Let 
 $\delta \bg = \bu \lvert_{\Gam} - \widetilde{\bg} $ be the error in the data.
The error $\delta f =  f -  f^{(M)} $ in the reconstructed source is estimated by 
\begin{align}\label{del_f}
 \lnorm{ \delta f }_{L^{2}(\OM)}^2 \leq (1+C)^{M+1} \left( c_1 \lnorm{L^{M+1}\bu}^2 +   c_2 \lnorm{L^{M} \delta \bg \lvert_{\Gam} }^2_{\frac{1}{2},\frac{1}{2}}  +c_3 \lnorm{\delta \bg \lvert_{\Gam} }^2  \right),
\end{align} where $c_1 = 96 (1+\mu)\sigma^4 \gamma_{M}^2$, $c_2 = 4\pi (1+\mu) \sigma^4$, $c_3 = \max \left\{2 \mu \sigma^4,2\right\} $, with $\sigma$ in \eqref{sigma_defn}, $\gamma_{M}$ in \eqref{p=1gamma},
and $C$ depends only on $\OM$ and $\ds \max \left\{ 1, \max_{0\leq j \leq M}    \|a- k_{-j}\|_{L^\infty(\Omega)}^2  \right\}$.
\end{theorem}

\begin{proof}

Since $u\in H^1(\OM\times\sph)$, $\bu  \in l^{2,1}(\BN;H^1(\OM))$, in particular $\bu \in l^{2,\frac{1}{2}}(\BN;H^1(\OM))$.

From linearity of the problem, the error $\delta f $ also satisfy \eqref{fMest}:
\begin{align}\label{delta_f}
 \lnorm{ \delta f }_{L^{2}(\OM)}^2 &\leq 2(1+C)^{M+1} \left( \lnorm{\delta u_{-M-1}}_{H^{1}(\OM)}^2 +   \lnorm{\delta u_{-M}}_{H^1(\Omega)}^2 + \sum_{j=1}^{M} \lnorm{\delta g_{-M+j}}_{H^{1/2}(\Gam)}^2  \right),
\end{align} where $\delta g_{m}$ for $ -M+1\leq m \leq 0$, is the $m$-th Fourier coefficient of $\delta \bg$.

 Recall that the $\delta u_{-M}$ and $\delta u_{-M-1}$ are the first two components of $e^{G}L^{M}\bq^{(M)}$. Using the gradient estimate in \eqref{estimateSimple} we have
\begin{align}\nonumber
 \lnorm{ L^{M} \del \bq^{(M)}}^2& \leq 12\|e^{-G}L^{M+1} K\bu\|_{1,0}^2+2\pi\|L^{M}\bq^{(M)} \lvert_{\Gam}\|^2_{\frac{1}{2},\frac{1}{2}}\\ \nonumber
 &\leq  48 \lnorm{\alpha}_{l^{1,1}_{\INF}(\ol \OM)}^2 \lnorm{L^{M+1}K\bu}_{1,0}^2 + 2\pi \lnorm{L^{M}\bq^{(M)} \lvert_{\Gam}}^2_{\frac{1}{2},\frac{1}{2}}\\ \label{cor_estimate}
 &\leq 48 \sigma^2 \gamma_{M}^2 \lnorm{L^{M+1}\bu}^2 + 2\pi \lnorm{L^{M}\bq^{(M)} \lvert_{\Gam}}^2_{\frac{1}{2},\frac{1}{2}},
\end{align}where the first inequality uses \eqref{eG_norm_onezero} in Proposition \ref{eGprop}, and the second inequality uses \eqref{LNp1} with $p=1$.
By Poincare and using \eqref{cor_estimate},
\begin{align}\nonumber
  \lnorm{ \delta u_{-M-1} }_{H^{1}(\OM)}^2 &+ \lnorm{ \delta u_{-M}}_{H^{1}(\OM)}^2 \leq \lnorm{e^{G} L^{M}\bq^{(M)}}^2_{0,1} 
  \leq \sigma^2 \lnorm{L^{M}\bq^{(M)}}^2_{0,1} \\ \nonumber
  &\leq (1+\mu)\sigma^2 \lnorm{ L^{M} \del \bq^{(M)}}^2 + \mu \sigma^2 \lnorm{L^{M} \bq^{(M)} \lvert_{\Gam}}^2 \\ \nonumber
  &\leq 48 (1+\mu)\sigma^4 \gamma_{M}^2 \lnorm{L^{M+1}\bu}^2 + 2\pi (1+\mu) \sigma^2 \lnorm{L^{M}\bq^{(M)} \lvert_{\Gam}}^2_{\frac{1}{2},\frac{1}{2}} + \mu \sigma^2 \lnorm{L^{M}\bq^{(M)} \lvert_{\Gam}}^2 \\  \label{delta_UMboth}
  &\leq 48 (1+\mu)\sigma^4 \gamma_{M}^2 \lnorm{L^{M+1}\bu}^2 + 2\pi (1+\mu) \sigma^4 \lnorm{L^{M} \delta \bg \lvert_{\Gam}}^2_{\frac{1}{2},\frac{1}{2}} + \mu \sigma^4 \lnorm{L^{M} \delta \bg \lvert_{\Gam}}^2.
\end{align}
From \eqref{delta_UMboth} and \eqref{delta_f}, the expression in \eqref{delta_f} becomes
\begin{align*}
 \lnorm{ \delta f}_{L^{2}(\OM)}^2 
 &\leq (1+C)^{M+1} \left( c_1 \lnorm{L^{M+1}\bu}^2 +   c_2 \lnorm{L^{M} \delta \bg \lvert_{\Gam}}^2_{\frac{1}{2},\frac{1}{2}} +2 \mu \sigma^4 \lnorm{L^{M} \delta \bg \lvert_{\Gam}}^2 
 +2 \sum_{j=1}^{M} \lnorm{ \delta g_{-M+j} }_{H^{1/2}(\Gam)}^2    \right)\\
 &\leq (1+C)^{M+1} \left( c_1 \lnorm{L^{M+1}\bu}^2 +   c_2 \lnorm{L^{M} \delta \bg \lvert_{\Gam}}^2_{\frac{1}{2},\frac{1}{2}}  +c_3 \lnorm{ \delta \bg\lvert_{\Gam}}^2  \right),
\end{align*} where constant  $c_1 = 96 (1+\mu)\sigma^4 \gamma_M^2$, $c_2 = 4\pi (1+\mu) \sigma^4$, $c_3 = \max \left\{2 \mu \sigma^4,2\right\} $, and $C$ depends only on $\OM$ and
$\ds \max \left\{ 1, \max_{0\leq j \leq M}  \lnorm{a- k_{-j}}_{L^\infty(\Omega)}^2  \right\}$.
Thus \eqref{del_f} follows.
\end{proof}

For the case when $k$ is a polynomial in the angular variable, $k_{-j} =0$ for $j \geq M+1$,  we obtain the following stability result.
\begin{cor}(Stability)\label{Cor_errorEst}
 Assume the hypotheses in Theorem \ref{error_estimate} with $k$ being a polynomial of degree $M$, i.e., $k_{-j} =0$ for $j \geq M+1$.
Then we have the following stability estimate  
\begin{align}\label{del_fkvanish}
\lnorm{ \delta f}_{L^{2}(\OM)}^2 \leq (1+C)^{M+1} \left( c_2 \lnorm{L^{M} \delta \bg \lvert_{\Gam}}^2_{\frac{1}{2},\frac{1}{2}}  +c_3 \lnorm{ \delta \bg \lvert_{\Gam}}^2  \right),
\end{align} where $c_2 = 4\pi (1+\mu) \sigma^4$, $c_3 = \max \left\{2 \mu \sigma^4,2\right\} $, with $\sigma$ in \eqref{sigma_defn}, $\gamma_{M}$ in \eqref{p=1gamma},
and $C$ depends only on $\OM$ and $\ds \max \left\{ 1, \max_{0\leq j \leq M}   \lnorm{a- k_{-j}}_{L^\infty(\Omega)}^2  \right\}$.
\end{cor}
\begin{proof}
This is the case $\gamma_{M}=0$ in \eqref{del_f}. So  $c_1=0$ and result \eqref{del_fkvanish} follows.
\end{proof}
We stress that if $k$ is not of polynomial type, then 
$ (1+C)^{M+1} \lnorm{L^{M+1}\bu}^2$ can not be made arbitrarily small. However, if  the anisotropic part of scattering is small enough (same case as in \cite{balTamasan07}), then we can prove a convergence result.

We have the following convergence result for the weakly anistropic scattering media. 
\begin{theorem}[]\label{convergence_result}
Let $a\in C^{2,s}(\ol \OM)$, $k\in C^{1,s}_{per}([-1,1]; Lip(\Omega) )$, for $s>1/2$, and $f\in H^1(\Omega)$.
Assume that  
\begin{align}\label{gamma_small}
\gamma = \sup_{j \geq 1} (1+j)^{p} \max \left\{ \lnorm{k_{-j}}_{\INF}, \lnorm{ \nabla_x k_{-j}}_{\INF} \right\}  < \frac{\sqrt{2} -1}{2\sqrt{\mu}\sigma^2},
\end{align} where $\mu$ as in Poincar\' e inequality \eqref{poincare} and $\sigma$ as in \eqref{sigma_defn}.\\
For each arbitrarily fixed $M$, let $f^{(M)}$ in \eqref{fsource_Scatpoly} be the reconstructed source using the 
data 
\begin{align}\label{data_fM}
 \widetilde{\bg} = (g_0, g_{-1}, \cdots , g_{-M+1}, g_{-M}, \widetilde{g_{-M-1}}, \widetilde{g_{-M-2}}, \cdots ),
\end{align} where $g_{-j}$, for $0\leq j\leq M$, are assumed exact data  and $\widetilde{g_{-j}}$, for  $j > M$, are allowed to be noisy.
Then 
\begin{align}
 \lnorm{ f^{(M)}  - f }^2_{L^2(\OM)} \to 0, \quad \text{as} \quad M \to \INF.
\end{align}

\end{theorem}

\begin{proof}
Since $f\in H^1(\Omega)$, the solution $u\in H^1(\OM\times\sph)$, and consequently $\bu  \in l^{2,1}(\BN;H^1(\OM))$, in particular $\bu \in l^{2,\frac{1}{2}}(\BN;H^1(\OM))$.

The error we make in the source reconstruction is controlled by the sequence valued map $\bq^{M} = e^{-G} (\bu-\bu^{M})\in l^{2,\frac{1}{2}}(\BN;H^1(\OM))$, which solves  
\begin{align}\label{bqM_BeltramiBMAM}
\dba\bq^{(M)} +L^2 \del\bq^{(M)} =  B^{M}\bq^{M} + A^M \bu,
\end{align} where $B^M$ and $A^M$ are operators given by 
\begin{align}\label{BMAM_defn}
B^M = e^{-G} L K^{(M)}e^G, \quad \text{and} \quad A^M = e^{-G} L (K- K^{(M)}),
\end{align} with $e^{\pm  G}$ as in \eqref{eGop}, $K$ as in \eqref{multiplier} and $K^{(M)}$ is the truncated Fourier multiplier as in  \eqref{KMmultiplier}.

The trace 
\begin{align}\label{q_boundary}
 \bq^{M}\lvert_{\Gam} = e^{-G} \left(\bu - \bu^{(M)}\right)\lvert_{\Gam} = e^{-G} \delta \bg \lvert_{\Gam}
\end{align} is determined by the error in the data
\begin{align}\label{delta_bg}
 \delta \bg \lvert_{\Gam} = \left(\bu - \bu^{(M)}\right)\lvert_{\Gam} =  (0,0,\cdots,0, g_{-M-1} - \widetilde{g_{-M-1}} , g_{-M-2} - \widetilde{g_{-M-2}}, \cdots ).
\end{align}

For  $c_1 = \max \left\{   \lnorm{(a-k_0)}^2_{\INF},2\right\} $, the norm 
\begin{align} \nonumber
	 \lnorm{f - f^{(M)} }^2_{L^2(\OM)}  &\leq \lnorm{(a-k_0)}^2_{\INF} \lnorm{u_0 - u^{(M)}_{0} }^2_{L^2(\OM)}  + 2 \lnorm{\del u_{-1} - \del u^{(M)}_{-1} }^2_{L^2(\OM)} \\ \nonumber
	  &\leq c_1  \left( \lnorm{ u_0 - u^{(M)}_{0} }_{H^{1}(\OM)}^2 + \lnorm{ u_{-1} - u^{(M)}_{-1} }_{H^{1}(\OM)}^2 \right)\\ \nonumber
	  &\leq c_1  \lnorm{e^{G} \bq^{(M)} }^2_{0,1} \\ \nonumber &\leq c_1 \sigma^2 \left( \lnorm{\bq^{(M)}}^2 + \lnorm{\del \bq^{(M)} }^2 \right)  \\ \nonumber
	  &\leq c_1 \sigma^2 (1+\mu) \lnorm{\del \bq^{(M)} }^2 + c_1\sigma^2 \mu \lnorm{\bq^{(M)} \lvert_{\Gam}}^2_{0,\frac{1}{2}},  \\ \label{fnorm_diff}
	  &\leq c_1 \sigma^2 (1+\mu) \lnorm{\del \bq^{(M)} }^2 + c_1\sigma^4 \mu \lnorm{ L^{M+1} (\bg- \widetilde{\bg}) }^2_{0,\frac{1}{2}},
\end{align} where in the fourth and last inequality we use Proposition \ref{eGprop}, in the next to last inequality we use the Poincar\' e inequality \eqref{poincare} and in the last inequality we have used \eqref{q_boundary}.\\
To estimate the gradient term $\lnorm{\del \bq^{(M)} }^2$ in \eqref{fnorm_diff}, we employ Theorem \ref{newidentity} with $\bv = \bq^M$, $B = B^M$, and $\bbf = A^M \bu$ in there as follows.
For $\bu \in l^{2,\frac{1}{2}}(\BN;H^1(\OM))$, we have 
 \begin{align*}
        \lnorm{A^M \bu}_{1,0} &\leq \sigma  \lnorm{ L (K-K^{(M)})\bu}_{1,0} = \sigma  \lnorm{ L^{M+1}K \bu}_{1,0} \leq \gamma \sigma \lnorm{L^{M+1} \bu}, 
   \end{align*} where the first inequality uses Proposition \ref{eGprop}, and the last inequality uses Lemma \ref{key} (iii) for $p=1$. 
Since  $\bq^{M} \in l^{2,\frac{1}{2}}(\BN;H^1(\OM))$, we have
   \begin{align*}
	\lnorm{B^M \bq^{M} }_{1,0} &\leq \sigma \lnorm{L K^{(M)} e^G \bq^{M}}_{1,0} \leq \gamma \sigma  \lnorm{e^G \bq^{M}} \leq \gamma \sigma^2  \lnorm{\bq^{M}}, \quad \text{and} \\
	\lnorm{B^M \bq^{M} }_{1,1} &\leq \sigma \lnorm{L K^{(M)} e^G \bq^{M}}_{1,1 } \leq \sqrt{2}\gamma \sigma \lnorm{ e^G\bq^{M}}_{0,1} \leq  \sqrt{2}\gamma \sigma^2  \lnorm{\bq^{M}}_{0,1},
   \end{align*} where the first inequality uses Proposition \ref{eGprop}, and the last inequality uses Lemma \ref{key} (iii) for $p=1$. 
We can apply Theorem \ref{newidentity} with $C_B = \sqrt{2}\gamma \sigma^2$ and $\epsilon  = \sqrt{2\mu}C_B < \sqrt{2} -1$ by \eqref{gamma_small}, to obtain
\begin{align*}
0\leq \lnorm{\del \bq^{(M)} }  \leq  \frac{b_M+\sqrt{b_{M}^2+4ac_M}}{2a},
\end{align*}
where
\begin{align*}
 &a = 1- 2\epsilon - \epsilon^2>0, \\ 
 &b_M = 2 \epsilon \lnorm{ \bq^{M} \lvert_{\Gam} }_{0,\frac{1}{2}} + 2 \sqrt{2} \gamma \sigma \lnorm{L^{M+1} \bu}, \\
 &c_M=\epsilon^2 \lnorm{ \bq^{M} \lvert_{\Gam} }_{0,\frac{1}{2}}^{2} +\pi \lnorm{ \bq^{M} \lvert_{\Gam} }^{2}_{\frac{1}{2},\frac{1}{2}} + 2 \gamma^2 \sigma^{2} \lnorm{L^{M+1} \bu}^2.
\end{align*}
Since the data is assumed exact for the first $M$ modes, $\ds (\bg- \widetilde{\bg})\lvert_{\Gam}$ in \eqref{delta_bg}  has the first $M$ modes equal to zero. Thus  
\begin{align*}
 \lnorm{ \bq^{M} \lvert_{\Gam} }_{0,\frac{1}{2}}^{2} &\leq \sigma^2 \lnorm{ L^{M+1} (\bg- \widetilde{\bg}) \lvert_{\Gam} }_{0,\frac{1}{2}}^{2} \to 0, \\
 \lnorm{ \bq^{M} \lvert_{\Gam} }_{\frac{1}{2},\frac{1}{2}}^{2} &\leq \sigma^2 \lnorm{ L^{M+1} (\bg- \widetilde{\bg}) \lvert_{\Gam} }_{\frac{1}{2},\frac{1}{2}}^{2} \to 0,
\end{align*} as $M \to \INF$.
Since $\bu$ is fixed,  the expression $ \lnorm{L^{M+1} \bu} \to 0$, as $M \to \INF$.
Therefore, both $b_M$ and $c_M$ tends to zero and 
  $\ds \lim_{M\to \INF} \lnorm{\del \bq^{(M)} } \to 0.$ 
Using \eqref{fnorm_diff}, the norm $\ds \lim_{M\to \INF} \lnorm{f - f^{(M)}}^2_{L^2(\OM)}\to 0$. 
\end{proof}

\section{Numerical Experiments}\label{Sec:numerics}

In this section we demonstrate the numerical feasibility of the method on two numerical examples. We focus on the numerical results, and leave their realization details for a separate discussion ~\cite{fujiwaraSadiqTamasan}. Although the theoretical results assume a source of square integrable gradient, we shall see below that the numerical reconstruction works even in the case of a discontinuous source.

The numerical experiments consider 
\begin{align*}
 \btheta\cdot\nabla u(z,\btheta) +a(z) u(z,\btheta) &=  \int_{\sph} k(z,\btheta\cdot\btheta')u(z,\btheta') d\btheta' + f(z) , \quad (z,\btheta)\in \OM\times\sph,
\end{align*} with the two dimensional Henyey-Greenstein (Poisson) model of scattering kernel
\begin{align}\label{poisson}
 k(z, \btheta\cdot\btheta') = \mu_{\text{s}}(z) \dfrac{1}{2\pi} \dfrac{1-g^2}{1-2g \btheta\cdot\btheta' + g^2},
\end{align} and the attenuation coefficient 
\begin{align*}
 a(z) = \mu_{\text{s}}(z) + \mu_{\text{a}}(z),
\end{align*} where $\mu_\text{s}$ is the scattering coefficient, and $\mu_\text{a}$ is the absorption coefficient.
The parameter $g$ in \eqref{poisson} models a degree of anisotropy, with $g=1$  for the ballistic regime and $g=0$ for the isotropic scattering regime.
In our numerical experiments we work with $g=1/2$ and  $\mu_{\text{s}}(z)=5$, the latter value shows that, on average, the particle scatters within $1/5$ units of length along the path.


Let $R  = (-0.25,0.5)\times(-0.15,0.15)$ be the rectangular region, and 
\begin{align*}
B_1 &= \{ (x,y) \:;\: (x-0.5)^2 + y^2 < 0.3^2 \}, \quad \text{and} \quad 
B_2 = \left\{ (x,y) \::\: \left(x+0.25\right)^2 + \left(y-\dfrac{\sqrt{3}}{4}\right)^2 < 0.2^2 \right\},
\end{align*} 
be the circular region inside the unit disc $\OM$ as shown in Figure~\ref{fig:inclusions}.

\begin{figure}[ht]
\begin{tikzpicture}[scale=2]
\draw [line width=2] (0,0) circle (1); 
  \node at (0.9,0.9) {$\Omega$};

\draw [fill=gray,line width=0] (-0.25,0.433013) circle (0.2); 

\draw [fill=gray,line width=0] (-0.25,-0.15) rectangle (0.5,0.15); 
  \node at (-0.3,-0.3) {$R$};

\draw [dotted,line width=2] (0.5,0) circle (0.3); 
  \node at (0.5,0.45) {$B_1$};

\draw [dotted,line width=2] (-0.25,0.433013) circle (0.2); 
    \node at (-0.25,0.8) {$B_2$};
\end{tikzpicture}
\caption{\label{fig:inclusions}
Locations of inclusions in numerical examples.
The solution $u(z,\tta)$ is relatively strongly absorbed inside the dotted balls,
while the source $f(z)$ is positive in the gray regions.
}
\end{figure}
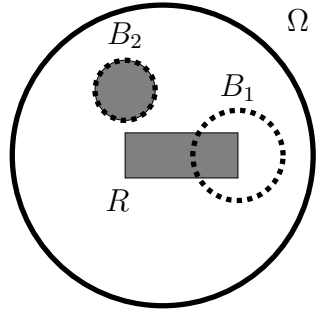


We consider two examples for the attenuation coefficient $a(z)$ for two different functions $ \mu_{\text{a}}(z)$, while $\mu_s=5$ is the same.
In the first example, we work with a $C^2$ smooth absorption, whereas in the second example we consider attenuation to be discontinuous. 
In  the first example, the attenuation $a$ is $C^2$-smooth via some scaled and translated quartic polynomial
$(|z|+1)^2(|z|-1)^2$, $0 \leq |z| \leq 1$, in the $\epsilon$-neighborhoods of $\partial B_1$ and $\partial B_2$: 

\begin{equation}\label{attenuation_1}
\mu_{\text{a}}(z) = \begin{cases}
2, \quad &\text{in $\bigl\{\: z \::\: \dist(z,\partial B_1) \ge \epsilon\bigr\}\cap B_1$};\\
1, \quad &\text{in $\bigl\{\: z \::\: \dist(z,\partial B_2) \ge \epsilon\bigr\}\cap B_2$};\\
0.1, \quad & \text{in $\bigl\{\: z \::\: \dist(z,B_1) \ge \epsilon\bigr\}
                       \cap \bigl\{ z \::\: \dist(z,B_2) \ge \epsilon\bigr\}$}
\end{cases}
\end{equation}
with  $\epsilon = 0.025$. 

In both examples the data is generated with the same source 
\begin{equation}\label{exact_source}
f(z) = \begin{cases}
2,  \qquad & \text{in $R$};\\
1,  \qquad & \text{in $B_2$};\\
0,  \qquad & \text{otherwise}.
\end{cases}
\end{equation}
This is the function we reconstruct by implementing the method in Section \ref{Sec:Sourcerec_poly}.

We generate the boundary measurement by the numerical computation
of the forward problem by piecewise constant approximation following the method in \cite{fujiwara}. In the forward problem the triangular mesh has
$1,966,690$ triangles,
while the velocity direction is split into $360$ equi-spaced intervals.

To obtain the data, we disregard the value of the solution inside and only keep the boundary values. The  boundary data, $u(z,\btheta)$ on $\partial\Omega\times S^1$, 
is represented in Figure~\ref{fig:smootha:boundary}:
\begin{figure}[ht]
\begin{minipage}{.6\textwidth}
\centering
\includegraphics[width=.8\textwidth]{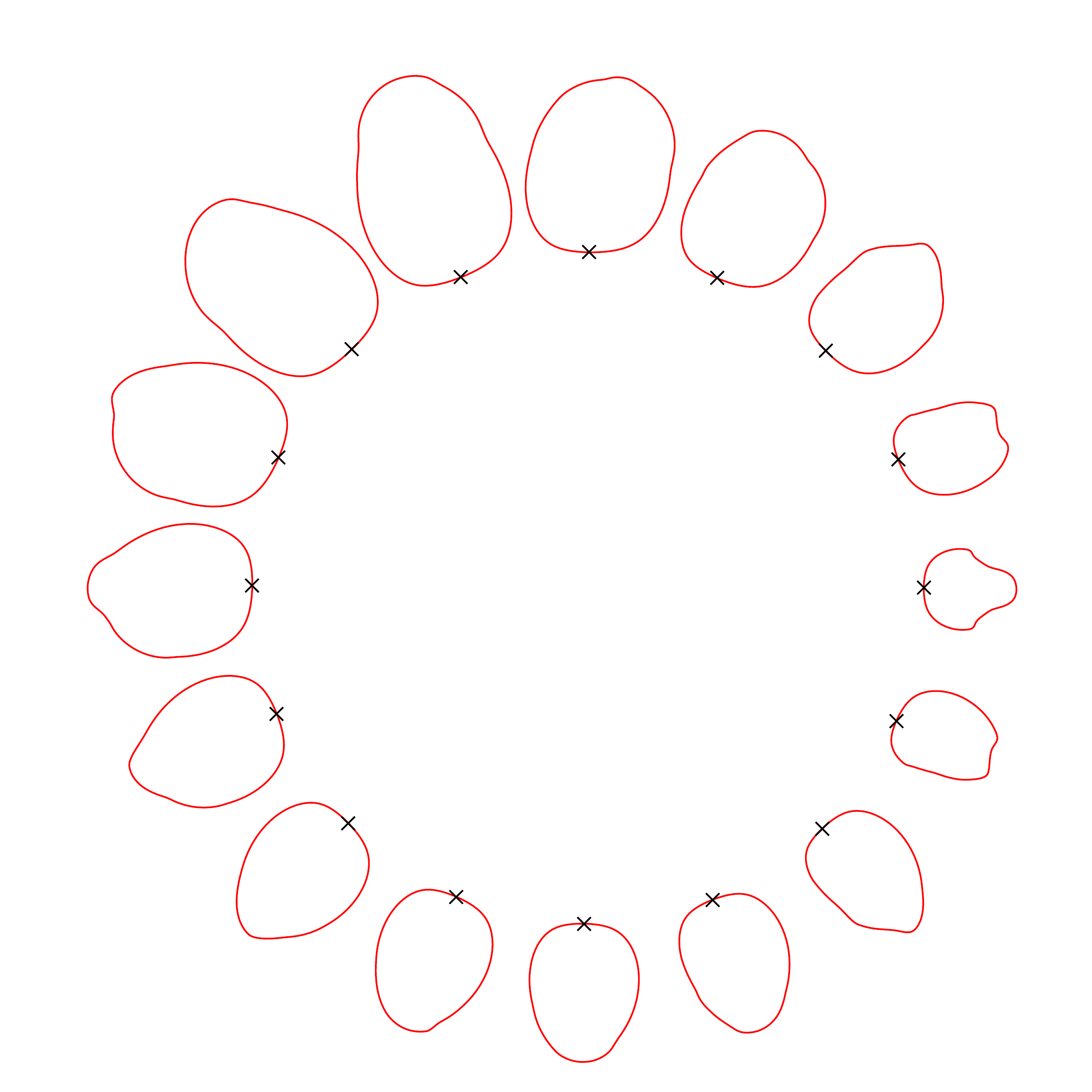}
\end{minipage}
\hfill
\begin{minipage}{.38\textwidth}
\centering
\includegraphics[width=\textwidth]{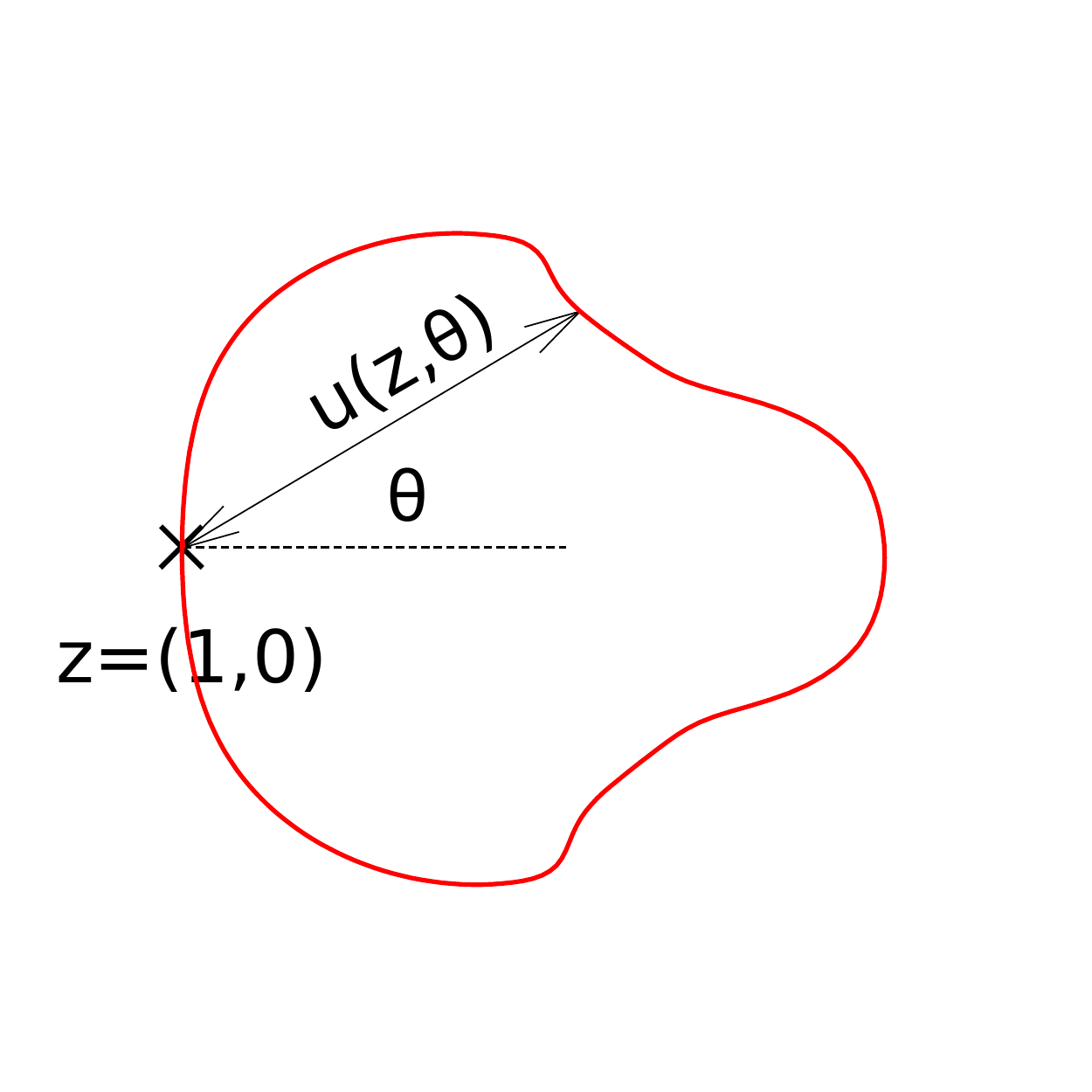}
\end{minipage}
\caption{\label{fig:smootha:boundary}Boundary measurement $u(z,\btheta)|_{\partial\Omega\times S^1}$
obtained by numerical computation of the corresponding forward problem.
The red curves are $\bigl\{z + 2u(z,\btheta)\btheta \::\: \btheta \in S^1\bigr\}$ for $z\in\partial\Omega$
indicated by the cross symbols ($\times$) (left).
The right figure is a magnification of that at $z=(1,0)$.
}
\end{figure}
For each $z \in \partial\Omega$ (indicated by a cross), we represent
the graph $\{\bigl(u(z,\btheta), \btheta\bigr):\;\btheta \in S^1\}$  as the (red) closed curve in the polar coordinates 
$\big\{ z + 2 u(z,\btheta)\btheta \::\: \btheta \in S^1\bigr\}$. Note that, since  $u|_{\Gamma_{-}} = 0$ (there is no incoming radiation),
the red curve lies outside the domain $\Omega$. The nearly tangential behavior at the boundary is the numerical evidence that the regime is far from ballistic, while the irregular shapes show that it is far from isotropic.

The reconstruction starts with solving \eqref{infMsys} with $M = 8$ via
the Bukhgeim-Cauchy integral formula \eqref{CauchyBukhgeimformula},
where the infinite series is replaced by a finite sums of up to $64$ terms.

In the reconstruction procedure,
all the steps, including the integral transforms $D[a]$, $R[a]$ and $H\bigl[R[a]\bigr]$,
have been processed numerically~\cite{fujiwaraSadiqTamasan}.
In solving the Poisson equation \eqref{Poisson_VM} and \eqref{VM_Gam},
the standard $P_1$ finite element method (constant and piecewise linear elements)
is employed.

The triangulation used in the reconstruction is different from that in the forward problem. In particular the reconstruction mesh consists of $6,998$ triangles (much less than the $1,966,690$ triangles used in the forward problem), and is generated without any information of the location of the subsets $R$, $B_1$, and $B_2$.

In the first numerical experiment the attenuation is given by \eqref{attenuation_1}. The reconstructed source $f(z)$ is shown in
Figure~\ref{fig:smootha:reconst} on the left.
On the right we show the cross section of $f(z)$
along the dotted diameter $y=-\sqrt{3}x$ passing through the origin and  the center of $B_2$. 

The reconstructed $f(z)$ shows a quantitative agreement with the exact source in \eqref{exact_source}. Similar to the attenuated X-ray tomography, the artifacts appear due to the co-normal singularities in the source, but also in the attenuation. This is because the most singular is the ballistic part (attenuated X-ray transform of the source). 
\begin{figure}[ht]
\centering
\begin{minipage}{.5\textwidth}
\includegraphics[width=\textwidth]{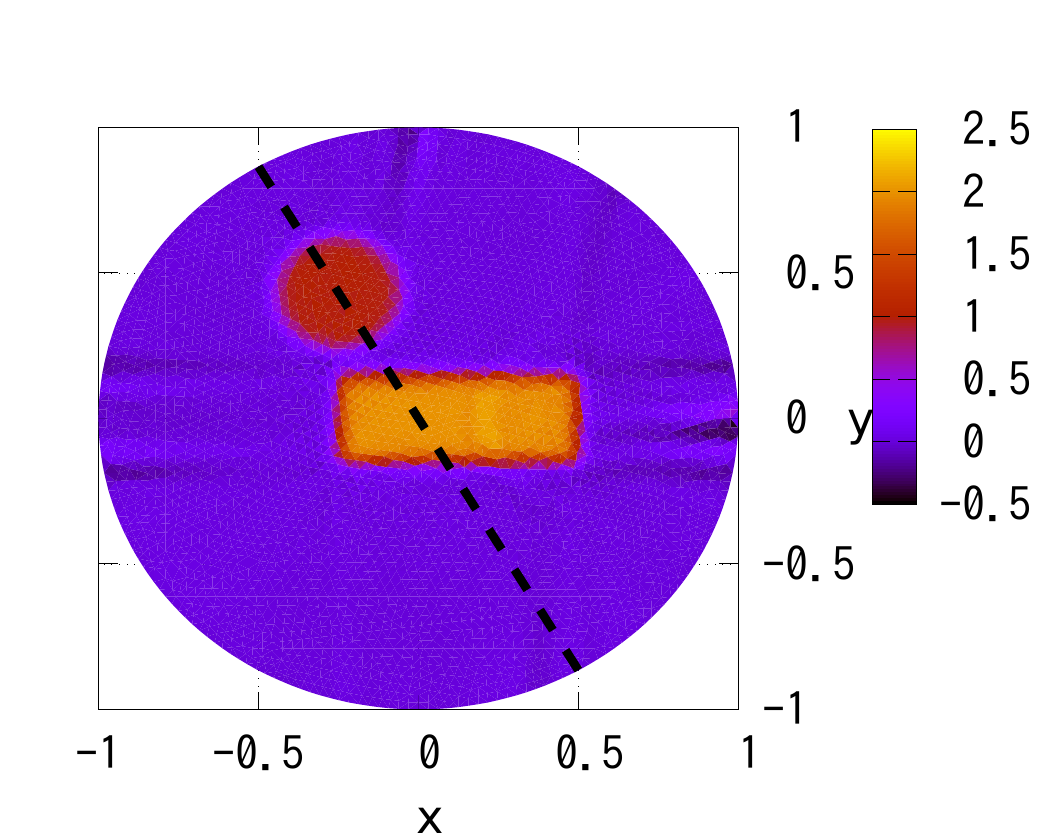}
\end{minipage}
\begin{minipage}{.45\textwidth}
\centering
\includegraphics[width=\textwidth]{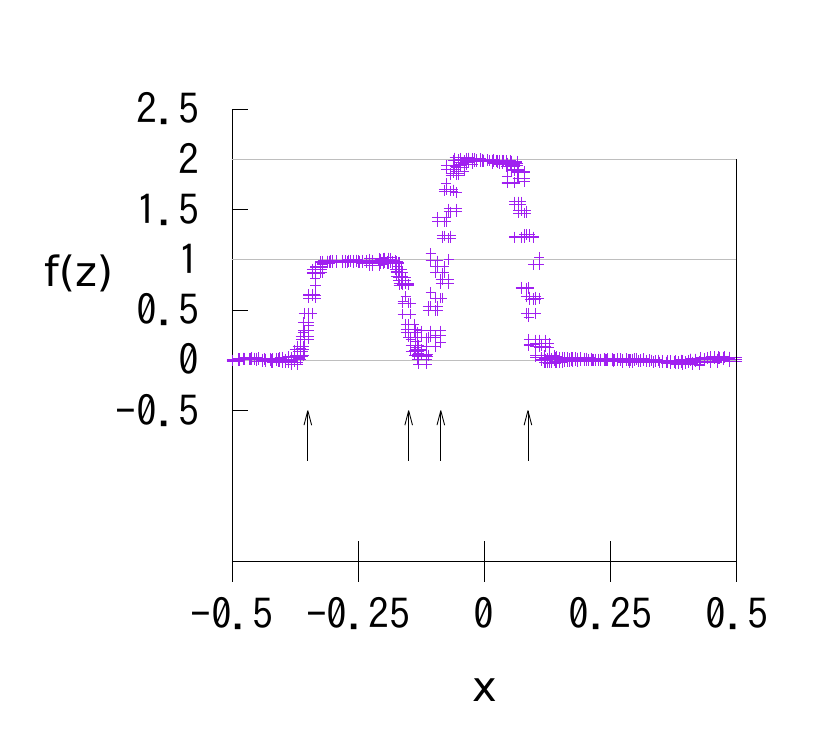}
\end{minipage}
\caption{\label{fig:smootha:reconst}$C^2$-smooth attenuation case: Reconstructed source (on the left)
and its section on the dotted line (on the right). The arrows on the right show the points where the dotted line meets  $\partial B_2$ and $\partial R$.}
\end{figure}

In the second example, we apply the proposed algorithm to the case of a discontinuous attenuation $a=\mu_\text{s}+\mu_\text{a}$ with \begin{equation*}
\mu_{\text{a}}(z) = \begin{cases}
2, \quad &\text{in $B_1$};\\
1, \quad &\text{in $B_2$};\\
0.1, \quad &\text{otherwise}.
\end{cases}
\end{equation*}

While different, the boundary data in the second experiment is graphically indistinguishable from the data in Figure~\ref{fig:smootha:boundary} in the smooth case.
\begin{figure}[ht]
\begin{minipage}{.5\textwidth}
\centering
\includegraphics[width=\textwidth]{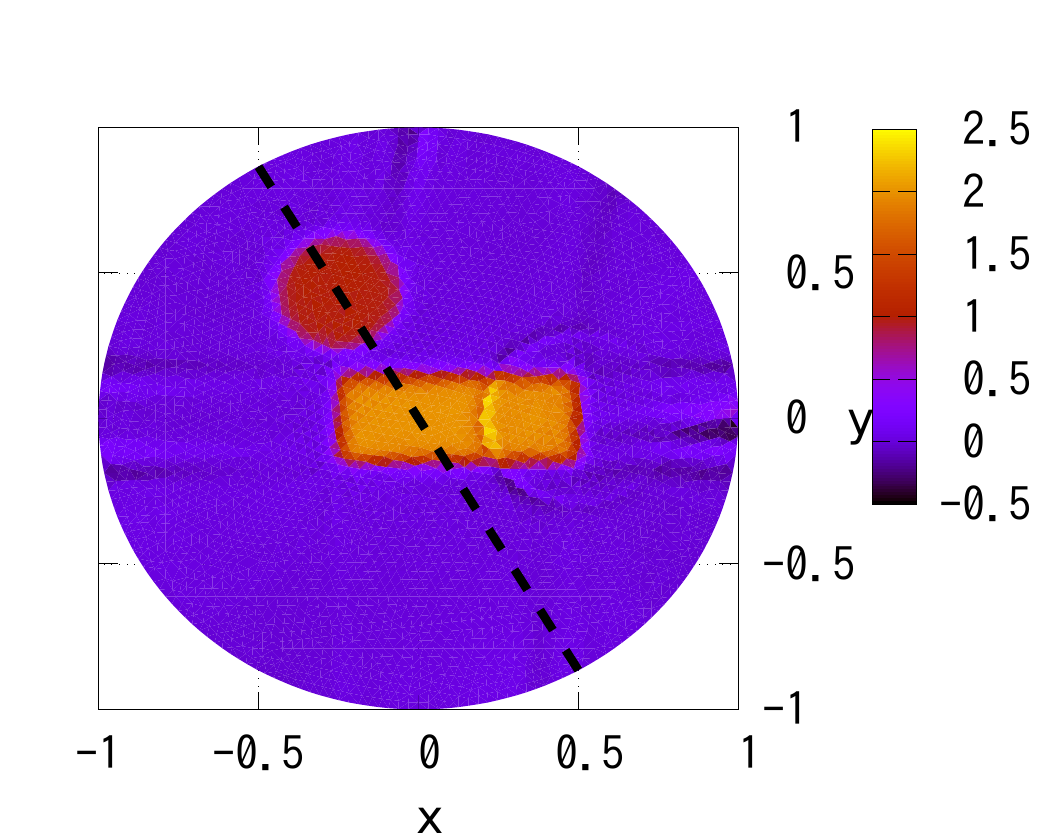}
\end{minipage}
\begin{minipage}{.45\textwidth}
\centering
\includegraphics[width=\textwidth]{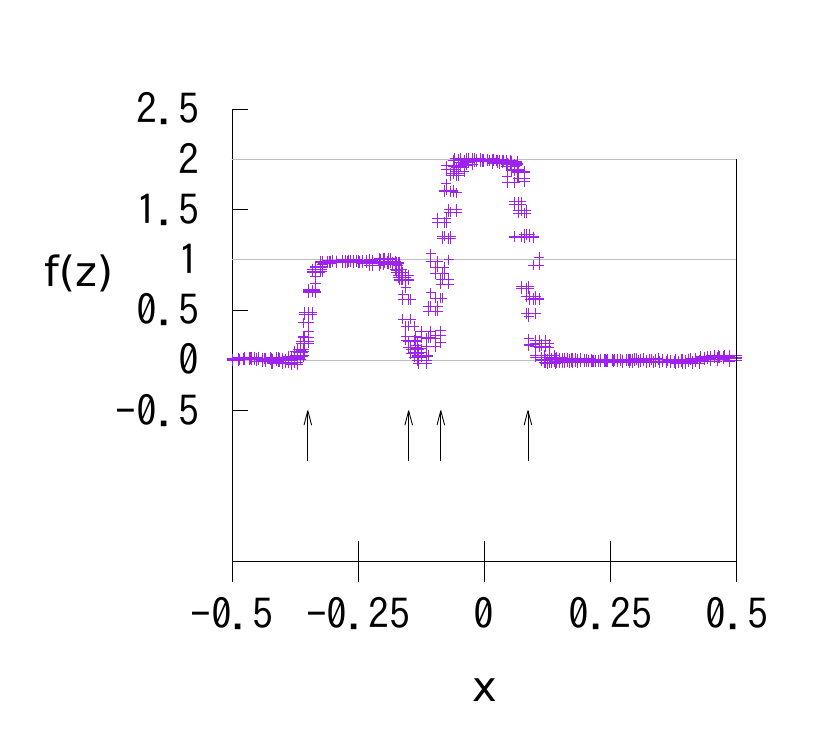}
\end{minipage}
\caption{\label{fig:disconta:reconst} Discontinuous attenuation case: Reconstructed source (on the left)
and its section on the dotted line (on the right).The arrows on the right show the points where the dotted line meets  $\partial B_2$ and $\partial R$.}
\end{figure}

The numerically reconstructed source, shown in Figure~\ref{fig:disconta:reconst}, agrees well with the exact source \eqref{exact_source},
implying that the proposed algorithm has a potential in reconstruction
even if the attenuation admits the discontinuity.

\begin{figure}[ht]
\begin{minipage}{.5\textwidth}
\centering
\includegraphics[width=\textwidth]{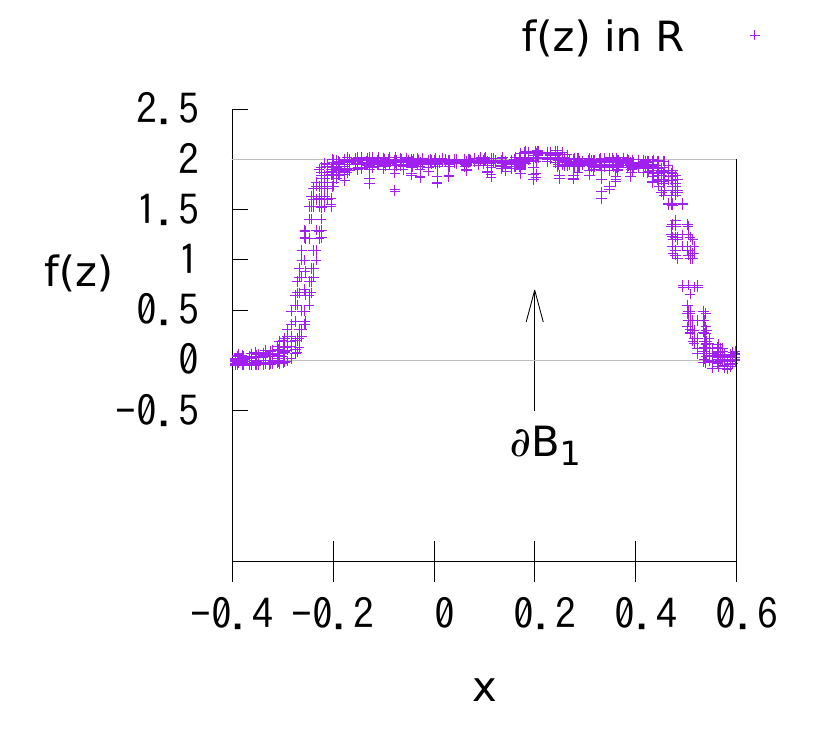}
\end{minipage}
\begin{minipage}{.45\textwidth}
\centering
\includegraphics[width=\textwidth]{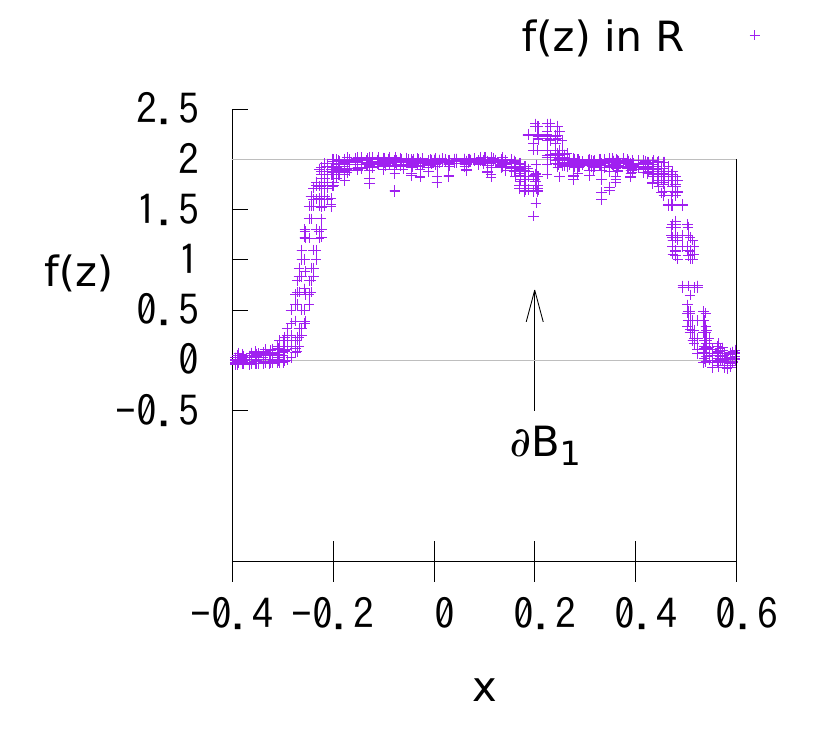}
\end{minipage}
\caption{\label{fig:projection}The projection on the $x$-axis of the reconstructed source $f$ in a neighborhood of the rectangle $R$ for the $C^2$-smooth attenuation (on the left) vs. discontinuous attenuation (on the right).
The arrows indicate $\partial B_1$.}
\end{figure}

Figure~\ref{fig:projection} shows
the projection on the $x$-axis of the reconstructed $f(z)$ for $z \in [-0.4,0.6]\times[-0.1,0.1]$. The arrows show the location of $\partial B_1$ in $R$, where the attenuation has a jump. 

In both reconstructions two types of artifacts appear, one type due to the singular support of $f$ and the other due to the singular support of $a$. Observe how the effect of the latter type of singularities diminishes with an increase in the smoothness of $a$.

At the level of singular support, the image is similar to that in the  non-scattering case ($k=0$). This is expected due to the smoothing effect of scattering.
A quantitative understanding of the effect of singularities in the attenuation, in the presence of scattering is subject to future work.

\section*{Acknowledgment}
The authors thank  D.~Kawagoe for useful discussions of regularity of solution of the forward problem.
The work of H.~ Fujiwara was supported by JSPS KAKENHI Grant Numbers 16H02155, 18K18719, and 18K03436.
The work of K.~ Sadiq  was supported by the Austrian Science Fund (FWF), Project P31053-N32. 

\section*{Appendix: New mapping properties of the integrating operator $e^{\pm G}$}
In this section we prove Proposition \ref{eGprop}.
\begin{proof}
Since $a\in C^{2,s}(\ol \OM)$, $s>1/2$, it follows from Lemma \ref{hproperties} that $e^{\pm h} \in C^{2,s}(\ol \OM \times \sph)$.
For $z \in \ol{\OM}$, let $\balpha(z) = \langle \alpha_{0}(z), \alpha_{1}(z),  ...  \rangle $, and $\del \balpha(z) = \langle \del \alpha_{0}(z), \del \alpha_{1}(z),  ...  \rangle $, where $\del$ is the derivative in the spatial variable.
By the proof in Bernstein's lemma \cite[Chap. I, Theorem 6.3]{katznelson},
\begin{align*}
 \underset{z \in \ol \OM}{\sup} \, \sum_{k=1}^{\INF}\, k \, \lvert \alpha_{k}(z) \rvert \leq  \underset{z \in \ol \OM}{\sup} \, \lnorm{e^{ -h(z,\cdotp)} }_{C^{2,s}}<\infty, \quad \text{and}
 \quad \underset{z \in \ol \OM}{\sup} \, \sum_{k=1}^{\INF}\, k \, \lvert \del \alpha_{k}(z) \rvert \leq  \underset{z \in \ol \OM}{\sup} \, \lnorm{ \del  e^{ -h(z,\cdotp)} }_{C^{1,s}}<\infty.
\end{align*}
Similarly, $\bbeta, \del \bbeta \in l^{1,1}_{\INF}(\ol \OM)$.

To prove the case $p=0=q$, let $\bu \in l^2(\BN;L^2(\OM))$.
By using Young's inequality
\begin{align} \nonumber
\lnorm{e^{-G}\bu}^2 = \sum_{j=0}^{\INF} \lnorm{ (e^{-G}\bu)_{-j} }^{2}_{ L^2(\OM)} &=  \int_{\OM} \lnorm{\balpha \ast \bu}^2_{l_2} \\ \label{esteGL2}
& \leq \int_{\OM} \lnorm{\balpha}^2_{l_1} \lnorm{\bu}^2_{l_2}   \leq \sup_{\ol \OM } \lnorm{\balpha}^2_{l_1} \int_{\OM}  \lnorm{\bu}^2_{l_2} = \sup_{\ol \OM } \lnorm{\balpha}^2_{l_1} \lnorm{\bu}^2,
\end{align} where the discrete convolution (with respect to the index) is defined in \eqref{eGop}.

To prove the case $p=1$ and $q=0$, let $\bu \in  l^{2,1}(\BN;L^2(\OM))$. We estimate the norm 
\begin{align}\nonumber
\lnorm{e^{-G}\bu}^2_{1,0} = \sum_{j=0}^{\INF} (1+j^2)  \lnorm{(e^{-G}\bu)_{-j}}^{2}_{ L^2(\OM)} &=  \sum_{j=0}^{\INF} \lnorm{(e^{-G}\bu)_{-j}}^{2}_{ L^2(\OM)}
+ \sum_{j=0}^{\INF} j^2 \lnorm{ (e^{-G}\bu)_{-j} }^{2}_{ L^2(\OM)}\\ \label{est_eGuL2} 
&\quad \leq \sup_{\ol \OM } \lnorm{\balpha}^2_{l_1} \int_{\OM}  \lnorm{\bu}^2_{l_2} + \int_{\OM} \sum_{j=0}^{\INF} j^2\lvert (e^{-G}\bu)_{-j} \rvert^{2},
\end{align} where in the first inequality we use estimate \eqref{esteGL2} from the case $p=0=q$.

To estimate the last term in \eqref{est_eGuL2}, we need to account for both positive and negative Fourier modes. 
Let $\widehat{u}$, and  $\widehat{e^{- h} }$ denote the sequence valued of the Fourier modes of $u$, and $e^{- h}$, respectively,  i.e.,
\begin{subequations} \label{seq_UeHhat}
\begin{align} \label{seq_Uhat}
 \widehat{u } = \langle \cdots, u_{-n}, \cdots , u_{-1}, &u_{0}, u_{1},\cdots, u_{n}, \cdots  \rangle,  \\ \label{seq_eHhat}
    \widehat{e^{- h}} = \langle \cdots, 0, \cdots \cdots, 0,\, &\alpha_{0}, \alpha_{1}, \cdots, \alpha_{n}, \cdots  \rangle.
\end{align}
\end{subequations}
By using Plancherel we estimate 
\begin{align}\nonumber
\int_{\OM} \sum_{j=0}^{\INF} j^2\lvert (e^{-G}\bu)_{-j} \rvert^{2} \leq \int_{\OM} \sum_{j=-\INF}^{\INF} j^2|(\widehat{e^{- h}} \ast \widehat{u })_j|^2 &=   \int_{\OM} \int_{\sph} \lvert\del_{\theta} (e^{-h} u) \rvert^{2}  \\ \label{est(ii)}
&\leq   \int_{\OM} \int_{\sph} \lvert e^{-h} \del_{\theta}  u \rvert^{2} +  \int_{\OM} \int_{\sph} \lvert u \, \del_{\theta} (e^{-h}) \rvert^{2}.
\end{align} 
To estimate the last two terms in \eqref{est(ii)}, let $\widehat{\del_{\theta}u }$ denote the (double) sequence of the Fourier modes of $\del_{\theta} u$, and
let $\widehat{\del_{\theta}e^{- h}}$ denote the sequence valued of the Fourier modes of $\del_{\theta}e^{- h}$, i.e., 
\begin{subequations} \label{seq_UeHDhat}
\begin{align} \label{seq_UDhat}
-i\widehat{\del_{\theta}u } = \langle \cdots, -nu_{-n}, \cdots , -u_{-1}, 0, &u_{1}, 2u_2, \cdots \cdots , nu_{n}, \cdots  \rangle,  \\ \label{seq_eHDhat}
   -i\widehat{\del_{\theta}e^{- h}}= \langle \cdots \cdots, 0, \cdots \cdots \cdots ,0, 0,\, & \alpha_{1}, 2\alpha_{2},\cdots \cdots , n\alpha_{n}, \cdots \rangle.
\end{align}
\end{subequations}
By using Plancherel the expression in \eqref{est(ii)} becomes 
\begin{align*}
\int_{\OM} \sum_{j=0}^{\INF} j^2\lvert (e^{-G}\bu)_{-j} \rvert^{2} 
&\leq  \int_{\OM} \lnorm{\widehat{e^{- h}} \ast \widehat{\del_{\theta}u} }^2_{l_2(\BZ)} + \int_{\OM} \lnorm{\widehat{\del_{\theta}e^{- h}} \ast \widehat{u} }^2_{l_2(\BZ)} \\
& \leq  \sup_{\ol \OM } \lnorm{ \widehat{e^{- h}}}^2_{l_1(\BZ)}  \int_{\OM}  \sum_{j=-\INF}^{\INF} |(\widehat{\del_{\theta}u})_j|^2 
+ \sup_{\ol \OM } \lnorm{\widehat{\del_{\theta}e^{- h}} }^2_{l_1(\BZ)} \int_{\OM}   \sum_{j=-\INF}^{\INF} |\widehat{u }_j|^2\\
& =  \sup_{\ol \OM } \lnorm{\balpha}^2_{l_1} \int_{\OM} \sum_{j=-\INF}^{\INF} j^2|\widehat{u}_j|^2 
+ \sup_{z \in \ol \OM } \left(\sum_{k=1}^{\INF}  k \lvert \alpha_{k}(z)\rvert \right)^2  \int_{\OM} \sum_{j=-\INF}^{\INF} |\widehat{u }_j|^2\\
& =  2\sup_{\ol \OM } \lnorm{\balpha}^2_{l_1} \int_{\OM} \sum_{j=0}^{\INF} j^2\lvert u_{-j} \rvert^{2}
+ 2 \lnorm{\balpha}^2_{l^{1,1}_{\INF}(\ol \OM)} \int_{\OM}  \lnorm{\bu}^2_{l_2},
\end{align*} where in the second inequality we use Young's  inequality, in the first equality we use $\ds \sup_{\ol \OM } \lnorm{ \widehat{e^{- h}} }^2_{l_1(\BZ)} = \sup_{\ol \OM } \lnorm{\balpha}^2_{l_1}$, and 
$\ds \sup_{\ol \OM } \lnorm{ \widehat{\del_{\tta}e^{- h}} }^2_{l_1(\BZ)} = \sup_{z \in \ol \OM } \left(\sum_{k=1}^{\INF}  k \lvert \alpha_{k}(z)\rvert \right)^2$, and in the last equality we use the norm in \eqref{lone1infdefn} and $u_{-n}=\ol{u_n}$ as $u$ is real valued.

From the above estimate of \eqref{est(ii)}, the expression in \eqref{est_eGuL2} becomes 
\begin{align}\nonumber
\lnorm{e^{-G}\bu}^2_{1,0}  &\leq  \sup_{\ol \OM } \lnorm{\balpha}^2_{l_1} \left( \int_{\OM}  \lnorm{\bu}^2_{l_2}  + 2 \int_{\OM}  \sum_{j=0}^{\INF} j^2\lvert u_{-j} \rvert^{2}  \right)
+ 2 \lnorm{\balpha}^2_{l^{1,1}_{\INF}(\ol \OM)}  \int_{\OM}  \lnorm{\bu}^2_{l_2} \\ \label{eG_10norm}
&\quad \leq 2 \sup_{\ol \OM } \lnorm{\balpha}^2_{l_1} \lnorm{\bu}^2_{1,0} + 2 \lnorm{\balpha}^2_{l^{1,1}_{\INF}(\ol \OM)} \lnorm{\bu}^2 \leq 4 \lnorm{\balpha}^2_{l^{1,1}_{\INF}(\ol \OM)} \lnorm{\bu}^2_{1,0}.
\end{align} 

To prove the case $p=0$ and $q=1$, let $\bu \in l^{2}(\BN;H^1(\OM))$, then $\bu , \del \bu \in l^{2}(\BN;L^2(\OM))$.
For $\bu \in l^{2}(\BN;L^2(\OM))$, $e^{\pm G} \bu \in l^{2}(\BN;L^2(\OM))$ from the case $p=0=q$, so it suffices to show that $\del(e^{-G} \bu) \in l^{2}(\BN;L^2(\OM))$. 

The discrete convolution below is with respect to the index, while $\del$ acts on the spatial variable.
\begin{align}\nonumber
\lnorm{\del(e^{-G} \bu)}^2  &=   \int_{\OM} \lnorm{\del(\balpha \ast \bu)}^2_{l_2} 
\leq \int_{\OM} \lnorm{\balpha \ast \del \bu}^2_{l_2} + \int_{\OM} \lnorm{\del\balpha \ast \bu}^2_{l_2}  \\ \label{est_1_{iii}}
&\quad \leq \int_{\OM} \lnorm{\balpha}^2_{l_1} \lnorm{\del \bu}^2_{l_2} +\int_{\OM} \lnorm{\del \balpha}^2_{l_1} \lnorm{\bu}^2_{l_2} 
\leq \sup_{\ol \OM } \lnorm{\balpha}^2_{l_1} \lnorm{\del \bu}^2 + \sup_{\ol \OM } \lnorm{\del \balpha}^2_{l_1} \lnorm{\bu}^2,
\end{align} where in the second inequality we use  discrete Young's  inequality. 

From \eqref{est_1_{iii}} and estimate \eqref{esteGL2} from the case $p=0=q$,
we estimate the norm 
\begin{align} \nonumber
\lnorm{e^{-G} \bu}^2_{0,1} &=  \lnorm{e^{-G} \bu}^2 + \lnorm{\del(e^{-G} \bu)}^2 
\leq \sup_{\ol \OM } \lnorm{\balpha}^2_{l_1}  \left( \lnorm{\bu}^2 + \lnorm{\del \bu}^2 \right) + \sup_{\ol \OM } \lnorm{\del \balpha}^2_{l_1} \lnorm{\bu}^2 \\  \label{eG_01norm}
& \;\leq \sup_{\ol \OM } \lnorm{\balpha}^2_{l_1}  \lnorm{\bu}^2_{0,1} + \sup_{\ol \OM } \lnorm{\del \balpha}^2_{l_1} \lnorm{\bu}^2  
\leq \left( \lnorm{\balpha}^2_{l^{1,1}_{\INF}(\ol \OM)} + \lnorm{\del \balpha}^2_{l^{1,1}_{\INF}(\ol \OM)} \right) \lnorm{\bu}^2_{0,1}.
\end{align}

To prove the case $p=1=q$, let $\bu \in l^{2,1}(\BN;H^1(\OM))$, then $\bu , \del \bu \in l^{2,1}(\BN;L^2(\OM))$.
Since $\bu \in l^{2,1}(\BN;L^2(\OM))$, $e^{- G} \bu \in l^{2,1}(\BN;L^2(\OM))$ by case $p=1$ and $q=0$. Thus suffices to show that $\del(e^{-G} \bu) \in l^{2,1}(\BN;L^2(\OM))$.   
We estimate the norm  
\begin{align} \nonumber
\lnorm{\del(e^{-G} \bu)}^2_{1,0}   &= \lnorm{\del(e^{-G} \bu)}^2 + \sum_{j=0}^{\INF} j^2 \lnorm{ \del(e^{-G}\bu)_{-j} }^{2}_{ L^2(\OM)} \\ \label{est_DeGL2}
&\quad \leq \sup_{\ol \OM } \lnorm{\balpha}^2_{l_1} \lnorm{\del \bu}^2 + \sup_{\ol \OM } \lnorm{\del \balpha}^2_{l_1} \lnorm{\bu}^2 + \int_{\OM} \sum_{j=0}^{\INF} j^2 \lvert (\del(\balpha \ast \bu))_{-j} \rvert^{2},
\end{align} where in the first inequality we have used estimate \eqref{est_1_{iii}} from the case $p=0$ and $q=1$.

To estimate the last term in \eqref{est_DeGL2}, we need to account for both positive and negative Fourier modes. 
Let $\widehat{u}$, and  $\widehat{e^{- h} }$ denote the sequence valued of the Fourier modes of $u$, and $e^{- h}$, respectively,  as in \eqref{seq_UeHhat}, and 
let $\widehat{\del_{\theta}u }$ and $\widehat{\del_{\theta}e^{- h}}$ denote the sequence valued of the Fourier modes of $\del_{\theta} u$, and $\del_{\theta}e^{- h}$, respectively,  as in \eqref{seq_UeHDhat}.
Moreover, let $\widehat{\del u }$, and  $\widehat{\del_{\theta} \del  u }$  denote the sequence valued of the Fourier modes of $\del  u$, respectively, $\del_{\theta} \del  u$, i.e.
\begin{align*}
 \widehat{\del  u } = \langle \cdots \cdots, \del u_{-n}, \cdots  \cdots , \del u_{-1}, \, &\del u_{0} \,, \del u_{1}, \cdots, \del u_{n}, \cdots  \rangle, \\
 -i\widehat{\del_{\theta} \del  u} = \langle \cdots, -n\del u_{-n}, \cdots , -\del u_{-1}, \, &\;0 \;, \; \del u_{1}, \cdots, n\del u_{n}, \cdots  \rangle.
\end{align*}
Similarly, let $\widehat{\del e^{- h}}$, and $\widehat{\del_{\theta}\del  e^{- h}}$  denote the sequence valued of the Fourier modes of $\del e^{- h}$, respectively, $\del_{\theta} \del  e^{- h}$, and recall that their negative Fourier modes vanish:
\begin{align*}
 \widehat{\del e^{- h}} = \langle \cdots, 0,\cdots 0,\del\alpha_0, \; &\del \alpha_{1}, \del \alpha_{2}, \cdots, \del\alpha_{n}, \cdots  \rangle, \\
 -i\widehat{\del_{\theta}\del  e^{- h}} = \langle \cdot \cdots, 0,\cdots ,\;  0  ,\; &\del \alpha_{1}, 2\del \alpha_{2},\cdots, n\del \alpha_{n}, \cdots  \rangle.
\end{align*}
Using Plancherel and the notations above, we estimate the last term in \eqref{est_DeGL2} as follows.
\begin{align}\nonumber
\int_{\OM} &\sum_{j=0}^{\INF} j^2 \lvert (\del(\balpha \ast \bu))_{-j} \rvert^{2}  \leq \int_{\OM} \sum_{j=0}^{\INF} j^2 \lvert (\del\balpha \ast \bu)_{-j} \rvert^{2} + \int_{\OM} \sum_{j=0}^{\INF} j^2 \lvert (\balpha \ast \del \bu))_{-j} \rvert^{2} \\ \nonumber
& \leq \int_{\OM} \sum_{j=-\INF}^{\INF} j^2 |(\widehat{\del  e^{- h}} \ast \widehat{u })_j|^2 + \int_{\OM} \sum_{j=-\INF}^{\INF} j^2 |(\widehat{e^{- h}} \ast \widehat{\del  u })_j|^2 \\ \nonumber
& = \int_{\OM} \int_{\sph} \lvert\del_{\theta} ((\del e^{-h}) u) \rvert^{2} + \int_{\OM} \int_{\sph} \lvert\del_{\theta} (e^{-h} \del  u) \rvert^{2} \\ \label{est(iii)}
& \leq \int_{\OM} \int_{\sph} \lvert (\del_{\theta} \del  e^{-h})   u \rvert^{2} +  \int_{\OM} \int_{\sph} \lvert \del (e^{-h}) \, \del_{\theta} u    \rvert^{2} 
+ \int_{\OM} \int_{\sph} \lvert \del_{\theta} e^{-h}   \del u \rvert^{2} +  \int_{\OM} \int_{\sph} \lvert e^{-h} \, \del_{\theta} (\del  u) \rvert^{2}.
\end{align}
Next we estimate each term on the right hand side of \eqref{est(iii)} separately.
\begin{enumerate}[I.]
 \item Estimate the first term in \eqref{est(iii)}: 
\begin{align*}
 \int_{\OM} \int_{\sph} \lvert (\del_{\theta} \del  e^{-h})   u \rvert^{2} &=  \int_{\OM} \lnorm{\widehat{\del_{\theta} \del  e^{- h}} \ast \widehat{u}}^2_{l_2(\BZ)}
 \leq \sup_{\ol \OM } \lnorm{\widehat{\del_{\theta} \del  e^{- h}} }^2_{l_1(\BZ)}  \int_{\OM} \lnorm{ \widehat{u}}^2_{l_2(\BZ)} \\
 &\quad = \sup_{\ol \OM } \left(\sum_{k=1}^{\INF}  k \lvert \del \alpha_{k}(z)\rvert \right)^2 \int_{\OM} \sum_{j=-\INF}^{\INF} |\widehat{u}_j|^2 = 2 \lnorm{\del \balpha}^2_{l^{1,1}_{\INF}(\ol \OM)} \int_{\OM} \lnorm{\bu}_{l_2}^2,
\end{align*} where in the first equality we use Plancherel, in the first inequality we use Young's  inequality,
in the second equality we have used the fact that $\sup_{\ol \OM } \lnorm{\widehat{\del_{\theta} \del  e^{- h}}}^2_{l_1(\BZ)} = \sup_{z \in \ol \OM } \left(\sum_{k=1}^{\INF}  k \lvert \del \alpha_{k}(z)\rvert \right)^2$,
and in the last equality we use the norm in \eqref{lone1infdefn} and the fact that $u_{-n}=\ol{u_n}$ as $u$ is real valued.

 \item Estimate the second term in \eqref{est(iii)}:
\begin{align*}
 \int_{\OM} \int_{\sph} \lvert \del  e^{-h} \, \del_{\theta} u \rvert^{2} &=  \int_{\OM} \lnorm{\widehat{\del  e^{-h}} \ast \widehat{\del_{\theta} u } }^2_{l_2(\BZ)}
 \leq \sup_{\ol \OM } \lnorm{\widehat{ \del  e^{- h}} }^2_{l_1(\BZ)}  \int_{\OM} \lnorm{ \widehat{\del_{\theta} u} }^2_{l_2(\BZ)} \\
 &\quad = \sup_{\ol \OM } \left(\sum_{k=0}^{\INF}  \lvert \del \alpha_{k}(z)\rvert \right)^2 \int_{\OM} \sum_{j=-\INF}^{\INF} |(\widehat{\del_{\theta }u})_j|^2 
 = 2\sup_{\ol \OM } \lnorm{\del \balpha}^2_{l_1} \int_{\OM} \sum_{j=0}^{\INF} j^2 \lvert u_{-j} \rvert^{2},
\end{align*} where in the first equality we have use Plancherel, in the first inequality we use Young's  inequality, 
in the second equality we have used the fact that $\sup_{\ol \OM } \lnorm{\widehat{\del  e^{- h}}}^2_{l_1(\BZ)} = \sup_{z \in \ol \OM } \left(\sum_{k=0}^{\INF}  \lvert \del \alpha_{k}(z)\rvert \right)^2$,
and in the last equality we use the fact that $u_{-n}=\ol{u_n}$ as $u$ is real valued.
 \item Estimate the third term  in \eqref{est(iii)}: 
\begin{align*}
 \int_{\OM} \int_{\sph} \lvert \del_{\theta} e^{-h}   \del u   \rvert^{2} &=  \int_{\OM} \lnorm{\widehat{\del_{\theta} e^{-h}} \ast \widehat{\del u}}^2_{l_2(\BZ)}
 \leq \sup_{\ol \OM } \lnorm{\widehat{ \del_{\theta} e^{- h}} }^2_{l_1(\BZ)}  \int_{\OM} \lnorm{ \widehat{\del  u} }^2_{l_2(\BZ)} \\
 &\quad = \sup_{\ol \OM } \left(\sum_{k=1}^{\INF}  k \lvert  \alpha_{k}(z)\rvert \right)^2 \int_{\OM} \sum_{j=-\INF}^{\INF} |(\widehat{\del  u})_j|^2 
 = 2 \lnorm{ \balpha}^2_{l^{1,1}_{\INF}(\ol \OM)}  \int_{\OM} \lnorm{\del \bu}_{l_2}^2,
\end{align*} where in the first equality we have use Plancherel, in the first inequality we use Young's  inequality, 
in the second equality we have used the fact that $\sup_{\ol \OM } \lnorm{\widehat{\del_{\theta} e^{- h}}}^2_{l_1(\BZ)} = \sup_{z \in \ol \OM } \left(\sum_{k=1}^{\INF} k \lvert \alpha_{k}(z)\rvert \right)^2$,
and in the last equality we use the norm in \eqref{lone1infdefn} and the fact that $u_{-n}=\ol{u_n}$ as $u$ is real valued.
\item Estimate the last term  in \eqref{est(iii)}:
\begin{align*}
 \int_{\OM} \int_{\sph} \lvert e^{-h} \, \del_{\theta} \del  u \rvert^{2} &=  \int_{\OM} \lnorm{\widehat{e^{- h}} \ast \widehat{\del_{\theta} \del  u } }^2_{l_2(\BZ)}
 \leq \sup_{\ol \OM } \lnorm{\widehat{e^{- h}} }^2_{l_1(\BZ)}  \int_{\OM} \lnorm{\widehat{\del_{\theta} \del  u} }^2_{l_2(\BZ)} \\
 &\quad = \sup_{\ol \OM } \left(\sum_{k=0}^{\INF}   \lvert  \alpha_{k}(z)\rvert \right)^2 \int_{\OM} \sum_{j=-\INF}^{\INF}  |(\widehat{\del_{\theta} \del  u})_j|^2 
        = 2\sup_{\ol \OM } \lnorm{\balpha}^2_{l_1} \int_{\OM} \sum_{j=0}^{\INF} j^2 \lvert \del u_{-j} \rvert^{2},
\end{align*} where in the first equality we use Plancherel, in the first inequality we use Young's  inequality, 
in the second equality we have used the fact that $\sup_{\ol \OM } \lnorm{\widehat{ e^{- h}}}^2_{l_1(\BZ)} = \sup_{z \in \ol \OM } \left(\sum_{k=0}^{\INF}  \lvert \alpha_{k}(z)\rvert \right)^2$,
and in the second to last equality we use $u_{-n}=\ol{u_n}$ as $u$ is real valued.
\end{enumerate}
The last term of \eqref{est_DeGL2} is thus estimated by
\begin{align*}\nonumber
\sum_{j=0}^{\INF} j^2\lVert (\del(e^{-G}\bu))_{-j} \rVert^{2}_{ L^2(\OM)} 
& \leq 2 \lnorm{\del \balpha}^2_{l^{1,1}_{\INF}(\ol \OM)} \int_{\OM} \lnorm{\bu}_{l_2}^2 + 2\sup_{\ol \OM } \lnorm{\del \balpha}^2_{l_1} \int_{\OM} \sum_{j=0}^{\INF} j^2 \lvert u_{-j} \rvert^{2} \\
&\qquad +  2 \lnorm{\balpha}^2_{l^{1,1}_{\INF}(\ol \OM)}  \int_{\OM} \lnorm{\del \bu}_{l_2}^2 + 2\sup_{\ol \OM } \lnorm{\balpha}^2_{l_1} \int_{\OM} \sum_{j=0}^{\INF} j^2 \lvert \del u_{-j} \rvert^{2}.
\end{align*}
Therefore the estimate \eqref{est_DeGL2} yields,
\begin{align*}
\lnorm{\del(e^{-G} \bu)}^2_{1,0} 
&\leq 2\lnorm{\del \balpha}^2_{l^{1,1}_{\INF}(\ol \OM)} \lnorm{\bu}^2 +  \sup_{\ol \OM } \lnorm{\del \balpha}^2_{l_1}  \left( \int_{\OM} \lnorm{\bu}_{l_2}^2 + 2\int_{\OM} \sum_{j=0}^{\INF} j^2 \lvert u_{-j} \rvert^{2} \right)\\
&\qquad + 2\lnorm{\balpha}^2_{l^{1,1}_{\INF}(\ol \OM)}  \lnorm{\del \bu}^2  +  \sup_{\ol \OM } \lnorm{\balpha}^2_{l_1}  \left( \int_{\OM} \lnorm{\del \bu}_{l_2}^2 + 2\sum_{j=0}^{\INF} j^2 \lvert \del u_{-j} \rvert^{2} \right) \\
& \leq 2  \lnorm{\balpha}^2_{l^{1,1}_{\INF}(\ol \OM)} \lnorm{\del \bu}^2 + 2\sup_{\ol \OM } \lnorm{\del \balpha}^2_{l_1} \lnorm{\bu}^2_{1,0}  + 2\lnorm{\balpha}^2_{l^{1,1}_{\INF}(\ol \OM)} \lnorm{\del \bu}^2 + 2 \sup_{\ol \OM } 
\lnorm{\balpha}^2_{l_1} \lnorm{\del \bu}^2_{1,0} \\
& \leq 4  \lnorm{\balpha}^2_{l^{1,1}_{\INF}(\ol \OM)} \lnorm{\del \bu}^2_{1,0}  + 4\lnorm{\balpha}^2_{l^{1,1}_{\INF}(\ol \OM)}  \lnorm{\del \bu}^2_{1,0} .
\end{align*}
Using the above estimate of $||\del(e^{-G} \bu)||^2_{1,0}$ and estimate \eqref{eG_10norm}, the norm  
\begin{align}\nonumber
\lnorm{e^{-G} \bu}^2_{1,1} =  \lnorm{e^{-G} \bu}^2_{1,0} + \lnorm{\del(e^{-G} \bu)}^2_{1,0} 
&\leq 4\left( \lnorm{\balpha}^2_{l^{1,1}_{\INF}(\ol \OM)} + \lnorm{\del \balpha}^2_{l^{1,1}_{\INF}(\ol \OM)} \right) \lnorm{\bu}_{1,0}^2 + 4\lnorm{\balpha}^2_{l^{1,1}_{\INF}(\ol \OM)} \lnorm{\del \bu}_{1,0}^2 \\ \label{eG_11norm}
&\; \leq 4\left( \lnorm{\balpha}^2_{l^{1,1}_{\INF}(\ol \OM)} + \lnorm{\del \balpha}^2_{l^{1,1}_{\INF}(\ol \OM)} \right) \lnorm{\bu}_{1,1}^2.
\end{align}

The mapping property for the case $p=\frac{1}{2}$ and $q=1$, follows from interpolation and the continuous embeddings 
$\ds
 l^{2,1}(\BN;H^1(\OM)) \xhookrightarrow{} l^{2,\frac{1}{2}}(\BN;H^1(\OM)) \xhookrightarrow{} l^{2}(\BN;H^1(\OM)).$
\end{proof}


\end{document}